\documentclass[11pt]{amsart}
\usepackage{verbatim,amssymb,amsmath,mathrsfs,graphicx,hyperref}
\usepackage{xcolor}
  \scrollmode

\newcommand{\R}{{\mathbb R}}

\newcommand{\N}{{\mathbb N}}
\newcommand{\Z}{{\mathbb Z}}
\newcommand{\EE}{{\mathbb E}}
\newcommand{\PP}{{\mathbb P}}
\newcommand{\C}{{\mathbb C}}

\newcommand{\Wc}{\mathcal W}
\newcommand{\Wcpi}{\mathcal W^\tau}
\newcommand{\Wtpi}{\widetilde W^\tau}
\newcommand{\Wtt}{\widetilde W^{\tau^*}}
\newcommand{\Ztpi}{\widetilde Z^\tau}

\newcommand{\Xtpi}{\widetilde X^\tau}

\newcommand{\mil}{Y^{\text{M}}}

\newcommand{\usn}{\underline {s}_n}
\newcommand{\utn}{\underline {t}_n}

\newcommand{\sgn}{\operatorname{sgn}}
\newcommand{\eps}{\varepsilon}
\newcommand{\teps}{\widetilde \varepsilon}

\newcommand{\Fc}{\mathcal{F}}

\newcommand{\tX}{\widetilde X}
\newcommand{\tY}{\widetilde Y}
\newcommand{\tW}{\widetilde W}

\newcommand{\bi}{\mathbf i}
\newcommand{\tpi}{\widetilde \tau}
\newcommand{\tPi}{\widetilde{\mathcal T}}

\newcommand{\tXp}{\widetilde X^{\tau}}
\newcommand{\tWp}{\widetilde W^{\tau}}
\newcommand{\tYp}{\widetilde Y^{\tau}}

\newcommand{\fh}{\widehat f}

\newcommand{\mus}{\mu_{\alpha, \beta}}
\newcommand{\fs}{f_{\alpha, \beta}}

\theoremstyle{plain}
\newtheorem{theorem}{Theorem}
\newtheorem{prop}{Proposition}
\newtheorem{lemma}{Lemma}

\theoremstyle{definition}

\begin{document}

\title[
Approximation of SDEs with a drift coefficient of H\"older or Sobolev regularity]{
	Sharp lower error bounds  for strong approximation of SDEs with a drift coefficient of H\"older or Sobolev regularity
	using a Weierstra{\ss} scale}

\author[Ellinger]
{Simon Ellinger}
\address{
	Faculty of Computer Science and Mathematics\\
	University of Passau\\
	Innstrasse 33 \\
	94032 Passau\\
	Germany} \email{simon.ellinger@uni-passau.de}

\author[M\"uller-Gronbach]
{Thomas M\"uller-Gronbach}
\address{
	Faculty of Computer Science and Mathematics\\
	University of Passau\\
	Innstrasse 33 \\
	94032 Passau\\
	Germany} \email{thomas.mueller-gronbach@uni-passau.de}

\author[Yaroslavtseva]
{Larisa Yaroslavtseva}
\address{
	Institute of Mathematics and Scientific Computing\\
	University of Graz\\
	Heinrichstra{\ss}e 36 \\
	8010 Graz\\
	Austria} \email{larisa.yaroslavtseva@uni-graz.at}

\begin{abstract} 
	In the present article, we study  strong approximation of solutions of scalar  stochastic differential equations (SDEs) with bounded and $\alpha$-H\"older continuous drift coefficient and constant  diffusion coefficient  at 
	time point $1$.
	Recently, it was  shown in~\cite{BDG21} that for such SDEs the equidistant Euler scheme achieves an $L^p$-error rate of at least $(1+\alpha)/2$, up to an arbitrary small $\varepsilon$, for all $p\geq 1$ and all $\alpha\in(0, 1]$,
	in terms of the number of evaluations of the driving Brownian motion $W$.
	In this article, we 
	prove a matching  lower error bound for $\alpha\in(0, 1)$.   More precisely, we show that for every $\alpha\in(0, 1)$,
	the $L^p$-error rate $(1+\alpha)/2$ of the Euler scheme in~\cite{BDG21} cannot be improved in general by  any numerical 
	method based on finitely many evaluations of $W$ at fixed time points
	in $[0,1]$. 
	Up to now, this result was known in the literature only for  $\alpha=1$. 
	Even stronger,
	we show that, for such SDEs, an 
	$L^p$-error rate better than $(1+\alpha)/2$ cannot, in general, be achieved, 
	even if in addition to the finitely many evaluations of 
	$W$, a finite number of integrals of  $W$ over fixed time intervals in $[0,1]$ may be used by an algorithm.
	In particular, Wagner–Platen type schemes are not superior to the Euler scheme with respect to the 
	$L^p$-error rate.
	
	Additionally, we extend a result from \cite{EMGY24} on sharp lower errror bounds for strong approximation of 
	SDEs with a bounded drift coefficient of fractional Sobolev regularity $\alpha\in (0,1)$ 
	and constant  diffusion coefficient  at 
	time point $1$.
	We prove that for every $\alpha\in (0,1)$, the $L^p$-error rate $ (1 + \alpha)/2$ 
	shown in \cite{DGL22} for the equidistant Euler scheme can, up to a logarithmic term, not be improved in general by any numerical method based on finitely many evaluations of $W$ 
	at fixed time points in $[0,1]$  
	and finitely many integrals of $W$ over fixed time intervals in $[0,1]$.
	This 
	lower bound
	was known from \cite{EMGY24} only for $\alpha\in (1/2,1)$, $p=2$ and numerical methods based on finitely many evaluations of $W$.
	
	Our results are the first lower bounds in the literature for the $L^p$-approximation of the solution $X$ of an SDE at a single time point by numerical methods based on finitely many evaluations as well as finitely many integrals of   $W$.  So far only lower bounds for the $L^p$-approximation of $X$ globally on the interval $[0,1]$ by such methods were known. 
	For the proof of our results we use variants of the Weierstrass function as a drift coefficient and we 
	extend
	the coupling of noise technique  recently introduced in~\cite{MGY23}
	to cover algorithms based on evaluations of $W$ as well as on integrals of $W$.

\end{abstract}

\maketitle

\section{Introduction and main results}\label{s3}
Consider a scalar additive noise driven  stochastic differential equation (SDE)
\begin{equation}\label{sde0}
	\begin{aligned}
		dX_t & = \mu(X_t) \, dt +  dW_t, \quad t\in [0,1],\\
		X_0 & = x_0
	\end{aligned}
\end{equation}
with deterministic initial value $x_0\in\R$, drift coefficient $\mu\colon\R\to\R$ and a one-dimensional driving
Brownian motion $W=(W_t)_{t\in[0,1]}$. Note that  the SDE~\eqref{sde0} has a unique strong solution if $\mu$ is  bounded and measurable, see~\cite{V80}.

In this article, we study the complexity of strong approximation of the solution $X$ of the SDE~\eqref{sde0} at the final time  $1$ by numerical methods based on finitely many evaluations and finitely many integrals of the driving Brownian motion $W$. 

We first consider the case where only finitely many evaluations of $W$ are used. A classical numerical method of this type is the Euler scheme  with $n$ equidistant  steps, given by 
$ X^{\text{E}}_{n,0} = x_0$ and
\[
X^{\text{E}}_{n,(i+1)/n} = X^{\text{E}}_{n,i/n}+ \mu\bigl( X^{\text{E}}_{n,i/n}\bigr)\cdot \frac{1}{n} + W_{(i+1)/n} -W_{i/n}
\]
for $i=1,\dots,n$.

Recently, it was proven  in~\cite{BDG21} that if the drift coefficient $\mu$ 
is bounded and $\alpha$-H\"older continuous with $\alpha\in(0,1]$, then the Euler scheme  achieves for all $p\in [1,\infty)$ an 
$L^p$-error
rate of at least $(1+\alpha)/2-$ in terms of the number of evaluations of $W$, i.e.,
for all $\eps\in (0,\infty)$ there exists $c\in (0,\infty)$ such that for all $n\in\N$,
\begin{equation}\label{eul}
	\EE\bigl[| X_1- X^{\text{E}}_{n,1}|^p\bigr]^{1/p} \le \frac{c}{n^{(1+\alpha)/2-\eps}}.
\end{equation}

This upper bound naturally leads to the question whether  $(1+\alpha)/2-$ is the best possible  $L^p$-error rate that can be achieved for approximating $X_1$ by 
numerical methods
based on finitely many evaluations of  $W$
at fixed time points in $[0,1]$, 
or whether there exists a method from this class that achieves  a better $L^p$-error rate than   $(1+\alpha)/2-$.

Up to now, the answer to this question was known in the literature only for $\alpha=1$.  
More precisely, it 
follows from more general results
in~\cite{m04} 
and ~\cite{hhmg2019} 
that if the
drift coefficient $\mu$  has bounded, continuous derivatives up to order $3$
on some open interval containing $x_0$
and satisfies 
\begin{equation}\label{cond3}
	\mu'(x_0)\neq 0,
\end{equation}
then 
the
best possible $L^1$-error rate that can be achieved by any numerical method based on finitely many evaluations of $W$
at fixed time points in $[0,1]$
is at most $1$, i.e., there exists $c\in (0,\infty)$ such that for all $n\in\N$,
\begin{equation}\label{lb5}
	\inf_{\substack{
			t_1,\dots ,t_n \in [0,1]\\
			g \colon \R^n \to \R \text{ measurable} \\
	}}	 \EE\bigl[|X_1-g(W_{t_1}, \ldots, W_{t_n})|\bigr]\geq  \frac{c}{ n}.
\end{equation}
The assumptions from~\cite{m04} 
and ~\cite{hhmg2019} 
are
in particular
satisfied for the SDE \eqref{sde0} with  $\mu=\cos$ and $x_0\in\R\setminus \{\pi k\mid k\in \mathbb Z\}$. Since this choice of
$\mu$ is bounded and Lipschitz continuous, we conclude that for all $p\in[1, \infty)$
the $L^p$-error rate  $(1+\alpha)/2-$ of the Euler scheme in \eqref{eul} can essentially  not be 
outperformed
in general by any  numerical method based on finitely many evaluations of $W$ if $\alpha=1$.

In the present article, we 
treat the case $\alpha\in (0,1)$. We show that for all
$\alpha\in (0,1)$
and  all $p\in[1, \infty)$ the $L^p$-error rate  $(1+\alpha)/2-$  of the Euler scheme in \eqref{eul}  
can essentially  not be 
improved
in general.
To be more precise,
for $\alpha\in
(0,1)$, let
\[
C^\alpha(\R)  = \biggl\{ f\colon \R\to\R\,\Bigl| \quad \sup_{x\not = y}\frac{|f(x)-f(y)|}{|x-y|^{\alpha}} <\infty \biggr\}
\]
be the 
space of  $\alpha$-H\"older continuous functions on $\R$. The 
first
result of this article is the following lower  bound.

\begin{theorem}\label{thm2}
	For
	every
	$\alpha \in (0,1)$
	there exist $c\in (0,\infty)$ and a bounded $\mu\in C^\alpha(\R)$ 
	such that  for all $n\in\N$,
	\begin{equation}\label{Mainlb}
		\inf_{\substack{
				t_1,\dots ,t_n \in [0,1]\\
				g \colon \R^n \to \R \text{ measurable} \\
		}}	 \EE\bigl[|X_1-g(W_{t_1}, \ldots, W_{t_n})|\bigr]\geq  \frac{c}{n^{(1+\alpha)/2}}.
	\end{equation}
\end{theorem}

Theorem \ref{thm2} also 
provides an essentially matching  lower  bound for the upper  bound recently proven  in~\cite{DG18} for
strong 
approximation of SDEs \eqref{sde0} with bounded and measurable drift coefficient $\mu$
by the Euler scheme.
Indeed, it was shown in~\cite{DG18} that for such SDEs the Euler scheme  
achieves for all $p\in [1,\infty)$ an 
$L^p$-error
rate of at least $1/2-$ in terms of the number $n$ of evaluations of $W$, i.e., 
for all $\eps\in (0,\infty)$ there exists $c\in (0,\infty)$ such that for all $n\in\N$,
\begin{equation}\label{eul2}
	\EE\bigl[| X_1-  X^{\text{E}}_{n,1}|^p\bigr]^{1/p} \le \frac{c}{n^{1/2-\eps}}.
\end{equation}
Choosing $\alpha=2\eps$ in Theorem \ref{thm2} we obtain the lower bound $c/n^{1/2+\eps}$, and thus the $L^p$-error rate $1/2-$ of the Euler scheme in \eqref{eul2} can essentially  not be 
improved
in general by any numerical method based on finitely many evaluations of $W$ for SDEs \eqref{sde0} with bounded and measurable drift coefficient $\mu$.
Up to now, 
it was only known that an $L^p$-error rate better than $3/4$ cannot be achieved  in general  for such SDEs, see~\cite{MGY23, ELL24, EMGY24}.

	Theorem \ref{thm2} is a consequence of the following more general result proven in this article.
	\begin{theorem}\label{thm2a}
		For
		every
		$\alpha \in (0,1)$
		there exist $c\in (0,\infty)$ and a bounded $\mu\in C^\alpha(\R)$ 
		such that  for all $n\in\N$,
		\begin{equation}\label{Mainlba}
			\inf_{\substack{
					t_1,\dots ,t_n \in [0,1]\\
					g \colon \R^{2n} \to \R \text{ measurable} \\
			}}	 \EE\Bigl[\Bigl|X_1-g\Bigl(W_{t_1}, \ldots, W_{t_n}, \int_0^{t_1} W_s\, ds, \ldots,
			 \int_{0}^{t_n} W_s \,ds\Bigr)\Bigr|\Bigr]\geq  \frac{c}{n^{(1+\alpha)/2}}.
		\end{equation}
	\end{theorem}
	Thus, in general, an $L^p$-error rate better than $(1+\alpha)/2$ cannot be achieved  for strong  approximation of $X_1$, even when  numerical methods based on finitely many evaluations  and finitely many integrals of  $W$ are used.

	A classical numerical method  of this type is the Wagner-Platen scheme  with $n$ equidistant  steps, given by 
	$ X^{\text{WP}}_{n,0} = x_0$ and
	\begin{align*}
		X^{\text{WP}}_{n,(i+1)/n} & = X^{\text{WP}}_{n,i/n}+ \mu\bigl( X^{\text{WP}}_{n,i/n}\bigr)\cdot \frac{1}{n}+ W_{(i+1)/n} -W_{i/n}+\Bigl(\frac{1}{2}\mu\mu'+\frac{1}{4}\mu''\Bigr)\bigl( X^{\text{WP}}_{n,i/n}\bigr)\cdot \frac{1}{n^2}\\
		& \qquad +\mu'\bigl( X^{\text{WP}}_{n,i/n}\bigr)\cdot\int_{i/n}^{(i+1)/n} (W_s-W_{i/n})\, ds
	\end{align*}
	for $i=1,\dots,n$. It is well known that if the
	drift coefficient $\mu$  has bounded derivatives up to order
	$3$ 
	then the Wagner-Platen scheme   achieves for all $p\in [1,\infty)$ an 
	$L^p$-error
	rate of at least $3/2$ in terms of the number of evaluations  and integrals of  $W$, i.e.,
	there exists $c\in (0,\infty)$ such that for all $n\in\N$,
	\[
	\EE\bigl[| X_1- X^{\text{WP}}_{n,1}|^p\bigr]^{1/p} \le \frac{c}{n^{3/2}},
	\]
	see e.g. \cite[Proposition V.4]{MG02_habil}. Together with the lower bound  \eqref{lb5},
	 this shows that, if $\mu$ additionally satisfies \eqref{cond3},  then the class of numerical methods based on finitely many evaluations and finitely many integrals of   $W$ outperforms the class of methods based solely on finitely many evaluations of  $W$ by at least  $1/2$ 
	in terms of the   $L^p$-error rate. In contrast, as follows from Theorem \ref{thm2a} and \eqref{eul}, for bounded and 
	$\alpha$-Hölder continuous drift coefficients $\mu$, the latter two classes of algorithms are equivalent with respect to the 
	$L^p$-error rate, and using Wagner-Platen type schemes does not help to improve this rate.

	Theorem \ref{thm2a} is the first lower bound in the literature for the $L^p$-approximation of $X_1$ by numerical methods based on finitely many evaluations and finitely many integrals of   $W$.  So far only lower bounds for the $L^p$-approximation of $X$ globally on the interval $[0,1]$ by such methods were known, see
	\cite{HMGR2002,HMG2004}.

In this article, we furthermore study the complexity of 
strong
approximation of $X_1$ in the case when the drift coefficient $\mu$ of the SDE \eqref{sde0} has
fractional Sobolev regularity. To be more precise, 
for $\alpha\in (0,1)$ and $p\in [1,\infty)$ let
\[
\mathsf W^{\alpha,p}(\R) = \biggl\{ f\colon \R\to\R\,\Bigl| \, f \text{ is measurable and }\int_\R\int_\R \frac{|f(x)-f(y)|^p}{|x-y|^{1+\alpha p}}\, dx\, dy <\infty \biggr\}
\]
be the
space of 
real-valued
functions $f$
on $\R$
that have Sobolev regularity of order $\alpha$ with integrability exponent $p$. 

In~\cite{DGL22} it was shown that if $\mu$ is bounded and  	$\mu \in\mathsf W^{\alpha,p}(\R)$ for some $\alpha\in (0,1)$ and  $p\in [1,\infty)$, then 
the Euler scheme 
with equidistant steps
achieves  an 
$L^p$-error
rate of at least $(1+\alpha)/2-$ in terms of the number $n$ of evaluations of $W$, i.e.,
for all $\eps\in (0,\infty)$ there exists $c\in (0,\infty)$ such that for all $n\in\N$,
\begin{equation}\label{eul3}
	\EE\bigl[| X_1- X^{\text{E}}_{n,1}|^p\bigr]^{1/p} \le \frac{c}{n^{(1+\alpha)/2-\eps}}.
\end{equation}

Recently, 
in~\cite{EMGY24} an essentially  matching lower bound for  $\alpha\in(1/2, 1)$ and $p=2$
was proven. 
More precisely, 
it was shown in~\cite{EMGY24} that
for
every
$\alpha\in (1/2,1)$
there exist $c\in (0,\infty)$ and a bounded,
Lebesgue integrable
$\mu\in\mathsf W^{\alpha,2}(\R)$ 
such that  for all $n\in\N$,
\[
\inf_{\substack{
		t_1,\dots ,t_n \in [0,1]\\
		g \colon \R^n \to \R \text{ measurable} \\
}}	 \EE\bigl[|X_1-g(W_{t_1}, \ldots, W_{t_n})|^2\bigr]^{1/2}\geq  \frac{c}{ \ln (n+1)\cdot n^{(1+\alpha)/2}}.
\]
Furthermore, it follows from~\cite{ELL24}  that for the 
SDE \eqref{sde0} with $\mu=1_{[0,1 ]}$ 
the best possible $L^p$-error rate that can be achieved by any numerical method based on finitely many evaluations of $W$ 
at fixed time points in $[0,1]$
is at most $3/4$ for all $p\in[1, \infty)$, i.e., there exists $c\in (0,\infty)$ such that for all $n\in\N$,
\[
\inf_{\substack{
		t_1,\dots ,t_n \in [0,1]\\
		g \colon \R^n \to \R \text{ measurable} \\
}}	 \EE\bigl[|X_1-g(W_{t_1}, \ldots, W_{t_n})|^p\bigr]^{1/p}\geq  \frac{c}{ n^{3/4}}.
\]
Since $1_{[0,1]}\in \mathsf W^{\alpha,2}(\R)$ for all $\alpha\in(0, 1/2)$, see e.g.
\cite[Section 3.1]{Sickel2021},
this result yields that the $L^p$-error rate  $(1+\alpha)/2-$ of the Euler scheme in \eqref{eul}  can 
essentially  not be 
improved
in general   for $\alpha=1/2$
and $p=2$.

However, 
the sharpness of \eqref{eul3}  in the case $\alpha<1/2$ or $p\not=2$
remained an open question up to now. In the present article, we close this gap and show
that  in fact for all  $\alpha\in (0,1)$ and  all $p\in[1, \infty)$ the $L^p$-error rate  $(1+\alpha)/2-$  of the Euler scheme in \eqref{eul3}  
can essentially  not be 
improved
in general, even when  numerical methods based on finitely many evaluations  and finitely many integrals of  $W$ are used. More formally, the following lower bound
is our second main result.

\begin{theorem}\label{thm3}
	For every $\alpha \in (0,1)$, every  $p \in [1,\infty)$ and every $\eps \in (0, \infty)$ there exist $c\in (0,\infty)$ and  a bounded $\mu\in \bigcap_{q\ge \min(p,2)} \mathsf W^{\alpha,q}(\R)$ with $\mu\in C^\alpha(\R)$ and $\mu\in \bigcap_{q\ge 1} L^q(\R)$
	such that  for all $n\in\N$,
	\begin{equation}\label{Mainlb_Sobolev}
		\begin{aligned}
			& 		\inf_{\substack{
						t_1,\dots ,t_n \in [0,1]\\
						g \colon \R^{2n} \to \R \text{ measurable} \\
				}}	 \EE\Bigl[\Bigl|X_1-g\Bigl(W_{t_1}, \ldots, W_{t_n}, \int_0^{t_1} W_s \,ds, \ldots,
				\int_{0}^{t_n} W_s \, ds\Bigr)\Bigr|\Bigr] \\
			&\qquad\qquad\qquad\qquad\qquad\qquad	\geq  \frac{c}{ (\ln (n+1))^{\gamma_p + \eps}\cdot n^{(1+\alpha)/2}},
		\end{aligned}
	\end{equation}
	where $\gamma_p=\tfrac{2}{\min(p,2)^2}$.
\end{theorem}

The proofs of Theorem \ref{thm2a} and Theorem \ref{thm3} are constructive. For
every
$\alpha \in (0,1)$,
a possible choice of 
the  drift coefficient $\mu$
in Theorem  \ref{thm2a} is given by the 
Weierstrass function 
\begin{equation}\label{muAlpha}
	\mu_\alpha(x) = \sum_{j = 1}^\infty 2^{-\alpha j} \sin(2^j x), \quad x \in \R,
\end{equation}
which is known to be bounded and $\alpha$-Hölder continuous, but not $(\alpha+\eps)$-Hölder continuous for any 
$\eps \in (0,\infty)$.
For every  $\alpha \in (0,1)$ and every $p\in[1, \infty)$,
a possible choice of
the  drift coefficient $\mu$ in Theorem  \ref{thm3} is given by the Weierstrass-type function 
\begin{equation}\label{muSob}
	\mu_{\alpha,\beta}(x) = 1_{[-2\pi, 4\pi]}(x) \cdot \sum_{j = 1}^\infty {j^{-\beta}} 2^{-\alpha j}  \sin(2^j x), \quad x \in \R, 
\end{equation}
with $\beta \in (\tfrac{1}{\min(p,2)}, \infty)$.

The rest of the  article is organised as follows. 
In Section \ref{S2}, we introduce some notation.
In Section~\ref{Trans}, we provide the construction and properties of a bi-Lipschitz transformation used to transform the SDE~\eqref{sde0} into an SDE with zero drift 
coefficient
and bounded Lipschitz continuous diffusion coefficient and we prove $L^p$-error estimates for a 
Milstein-type
approximation of the solution of the transformed SDE. Section~\ref{prelim} contains preliminary estimates that are used for both the proof of Theorem~\ref{thm2a} and the proof of Theorem~\ref{thm3}. 
The proofs of Theorem~\ref{thm2a} and Theorem~\ref{thm3} are then  carried out in Section \ref{proofthm2} and Section~\ref{proofthm3}, respectively. In the Appendix we provide an
auxiliary result on Brownian bridge processes as well as a technical moment estimate and we prove  properties of the Weierstrass-type function $\mu_{\alpha,\beta}$.

\section{Notation}\label{S2}

For a set $A\subset\R$ and a function $f\colon A\to\C$ we put $\|f\|_\infty = \sup_{x\in A} |f(x)|$. We use $\bi$ to denote the imaginary unit in $\C$. 
For $f\colon \R\to\R$
 with $f_{|[0,2\pi]} \in 
L^2([0,2\pi])$
and $j\in\Z$ we use 
\[
\fh_j = \frac{1}{\sqrt{2\pi}}\int_0^{2\pi} f(x) \exp(-\bi j x)\, dx,\quad j\in \Z,
\]
to denote the
$j$-th
 Fourier coefficient of $f$.

\section{The transformation}\label{Trans}
In this section we provide  the construction and properties of the bi-Lipschitz transformation used to transform the SDE~\eqref{sde0}
into an SDE with zero drift coefficient.
This transformation is a well-known tool to remove the drift coefficient of an SDE, see e.g.~\cite{NS19}.

Let $\mu\colon \R\to \R$ be locally integrable and define 
\begin{equation}\label{trans}
	G_\mu\colon \R\to\R, \, \, x\mapsto \int_0^x e^{-2\int_0^y \mu(z)\, dz}\, dy.
\end{equation}

The following lemma is a 
slight
 generalization of~\cite[Lemma 1]{EMGY24}.

\begin{lemma}\label{lem1NEW}
	Let $\mu\colon \R \to \R$ be  locally integrable  such that $\sup_{y \in \R} |\int_0^y \mu(z) \, dz| < \infty$. Then $G_\mu$ has the following properties.\\[-.3cm]
		\begin{itemize}
			\item [(i)] 
			$G_\mu$ is continuously differentiable
			and there exist $c_1,c_2\in (0,\infty)$ such that $c_1 \le G_\mu'\le c_2$.\\[-.3cm]
			\item[(ii)] $G_\mu$ is a bijection,
				 $G_\mu^{-1}$ is continuously differentiable
			and $c_2^{-1} \le (G_\mu^{-1})'\le c_1^{-1}$.\\[-.3cm]
		\end{itemize}
		If, additionally, $\mu$ is bounded then
		\begin{itemize}
			\item[(iii)]
	the functions $G_\mu'$, $G_\mu'\circ G_\mu^{-1}$ and  $(G_\mu^{-1})'$ are Lipschitz continuous with weak derivatives $-2\mu G_\mu'$, $-2\mu\circ G_\mu^{-1}$ and $2(\mu/(G_\mu')^2) \circ  G_\mu^{-1}$, respectively.
		\end{itemize}
\end{lemma}

\begin{proof}
 By the assumptions on $\mu$, the mapping 
\[
T\colon \R\to\R,\,\, y\mapsto  \int_0^y \mu(z)\, dz 
\]
is continuous and bounded. As a consequence, $G_\mu$ is continuously differentiable with 
\[
G_\mu'(x) = e^{-2T(x)},\,\, x\in \R.
\] 
Moreover, for every $x\in \R$,
\[
e^{-2\|T\|_{\infty}} \le e^{-2T(x)} \le 	e^{2\|T\|_\infty}, 
\]
which completes the proof of part (i). Part (ii) is an immediate consequence of part (i).

Next, assume, additionally, that $\mu$ is bounded. Then $T$ is Lipschitz continuous and, by the fundamental theorem of calculus, $T$ is differentiable Lebesgue-almost everywhere with weak derivative $\mu$. Since the mapping $S\colon\R\to \R$, $y\mapsto e^{-2y}$ is continuously differentiable and $T$ is bounded, we obtain that $G_\mu'=S\circ T$ is Lipschitz continuous with weak derivative $(S'\circ T) \cdot \mu = -2G_\mu'\cdot \mu$. This proves the statement on $G_\mu'$ in part (iii). 

By (ii) and the Lipschitz continuity of $G_\mu'$, we obtain the Lipschitz continuity of $G_\mu'\circ G_\mu^{-1}$. Furthermore, there exists a Borel set $A\subset \R$ such that $\lambda(A^c) =0$ and for every $x\in A$, the function $G'_\mu$ is differentiable in $x$ with derivative
$G_\mu''(x) = -2\mu(x)\cdot G_\mu'(x) $. We conclude that for every $x\in G_\mu(A)$, the function $G_\mu'\circ G_\mu^{-1}$ is differentiable in $x$ with derivative 
\[
(G_\mu'\circ G_\mu^{-1})'(x)=G_\mu''(G_\mu^{-1}(x))\cdot (G_\mu^{-1})'(x) = \frac{-2\mu(G_\mu^{-1}(x))G_\mu'(G_\mu^{-1}(x)) }{G_\mu'(G_\mu^{-1}(x))} = -2 \mu(G_\mu^{-1}(x)).
\]
By (i) and (ii) we obtain that $\lambda((G_\mu(A))^c)= \lambda(G_\mu(A^c)) = \int_{A^c} G_\mu'(x)\lambda (dx) =0$, which finishes the proof of the statement on $G_\mu'\circ G_\mu^{-1}$ in part (iii). 

By (i) we obtain that for all $x,y\in \R$,
\begin{align*}
	|(G_\mu^{-1})'(x) - (G_\mu^{-1})'(y)| &
	= \biggl|\frac{G_\mu'(G_\mu^{-1}(y)) - G_\mu'(G_\mu^{-1}(x))}{G_\mu'(G_\mu^{-1}(x))G_\mu'(G_\mu^{-1}(y))}\biggr| 
	%T wrong correction: 	\le c_2^{-2} 
	\le c_1^{-2} 
	|G_\mu'\circ G_\mu^{-1}(y) - G_\mu'\circ G_\mu^{-1}(x)|,
\end{align*}
which jointly with the Lipschitz continuity of $G_\mu'\circ G_\mu^{-1}$ yields the Lipschitz continuity of $(G_\mu^{-1})'$. Finally,  for every $x\in G_\mu(A)$, the function
 $(G_\mu^{-1})' = 1/(G_\mu'\circ G_\mu^{-1})$ 
 is differentiable in $x$ with derivative 
\[
(G_\mu^{-1})''(x) = -\frac{ (G_\mu'\circ G_\mu^{-1})'(x)}{(G_\mu'\circ G_\mu^{-1}(x))^2}  = \frac{2\mu}{(G_\mu')^2}(G_\mu^{-1}(x)),
\]
which completes the proof of part (iii) and hereby finishes the proof of the lemma.
\end{proof}

Applying $G_\mu$ to the solution $X$ of the SDE~\eqref{sde0} yields the solution $Y$ of an SDE with zero drift coefficient  and  Lipschitz continuous diffusion coefficient. The following result is a slight generalization of \cite[Lemma 2]{EMGY24}.

\begin{lemma}\label{lem2NEW}
	Let $\mu\colon\R \to \R$ be  measurable and bounded with $\sup_{y \in \R} |\int_0^y \mu(z) \, dz| < \infty$.
	Then 
	the SDE~\eqref{sde0} has a unique strong solution $X$ and the
	stochastic
	process                         
	\[
	Y=\bigl(Y_t = G_\mu(X_t)\bigr)_{t\in[0,1]}
	\]
	is 
	the  
	unique strong solution of the SDE 
	\begin{equation}\label{sde1NEW}
		\begin{aligned}
			dY_t & = b_\mu(Y_t) \,   dW_t, \quad t\in [0,1],\\
			Y_0 & = G_\mu(x_0),
		\end{aligned}
	\end{equation}
	where $b_\mu = G_\mu'\circ G_\mu^{-1}$ is Lipschitz continuous with weak derivative $b_\mu' = -2\mu\circ G_\mu^{-1}$.
\end{lemma}

The proof of Lemma \ref{lem2NEW} is almost identical to the proof of \cite[Lemma 2]{EMGY24} using Lemma \ref{lem1NEW} in place of \cite[Lemma 1]{EMGY24}.
For convenience of the reader we provide a proof of it here.
\begin{proof}
	By assumption, $\mu$ is measurable and bounded, which implies existence and uniqueness of a strong solution of~\eqref{sde0}, see
	~\cite{V80}. By Lemma~\ref{lem1NEW}(i),(iii) we may apply a generalized It\^{o} formula, see e.g.~\cite[Problem 3.7.3]{ks91}, to conclude that the 
	stochastic
	process $(G_\mu(X_t))_{t\in[0,1]}$ is a strong solution of the SDE~\eqref{sde1NEW}. The properties of $b_\mu$ are stated in Lemma~\ref{lem1NEW} (iii). As a consequence of the Lipschitz continuity of $b_\mu$, the strong solution of the SDE~\eqref{sde1NEW} is unique.
\end{proof}

For $n\in\N$  let  $\mil_n=(\mil_{n,t})_{t\in [0,1]}$ 
denote the 
 time-continuous Milstein-type scheme with step-size $1/n$ associated to the SDE~\eqref{sde1NEW}  given by
$\mil_{n, 0} =G_\mu(x_0)$ and 
\begin{equation}\label{mil}
	\mil_{n, t} = \mil_{n, \ell/n} + b_{\mu}(\mil_{n, \ell/n})(W_{t} - W_{\ell/n}) + \frac{1}{2} b_{\mu} b_{\mu}'(\mil_{n, \ell/n})((W_{t} - W_{\ell/n})^2 - (t-\ell/n))
\end{equation}

for all $t\in (\ell/n,(\ell+1)/n]$ and 	every $\ell \in \{0,1,\dots,n-1\}$.
Next, we provide an $L^p$-error estimate for $\mil_n$.

%L in the journal version: add a sentence that one can approximate Y by G_\mu(\eul) and use the upper bound of Dareiotis et al, but one would loose epsilon
%T such a comment is not understandable to the reader at this place, more useful in the sketch of proof tbd

The following result is a slight generalization of Theorem 7.5(a) in \cite{Pages2018}.

\begin{prop}\label{prop1NEW}
	Let $\alpha \in (0,1)$
	and let $\mu  \in C^\alpha(\R)$ be bounded with $\sup_{y \in \R} |\int_0^y \mu(z) \, dz| < \infty$.
	Then for all $p \in (0, \infty)$ 
	there exists
	$c \in (0, \infty)$ such that for all $n \in \N$,
	\begin{equation}\label{milest}
		\EE\bigl[\,\|Y - \mil_n\|_\infty^p\bigr]^{1/p} \leq \frac{c}{ n^{(1 + \alpha) / 2}}.
	\end{equation}
\end{prop}

\begin{proof}
	%T added
	We proceed similar to the proof of Theorem 7.5(a) in \cite{Pages2018}.
	Without loss of generality, we may assume that $p\ge 2$.
	For  $t\in [0,1]$ and $n\in \N$  put $\utn = \lfloor nt \rfloor/n$.
Then for all $t\in [0,1]$ and all $n\in \N$ we have
\begin{equation}\label{mil1}
	\begin{aligned}
		Y_t - \mil_{n,t} & = \int_0^t \bigl(b_\mu(Y_s) - b_\mu(\mil_{n,\usn}) - b_\mu b_\mu'(\mil_{n,\usn})(W_s-W_{\usn} )\bigr) \, dW_s\\
		& = \int_0^t (b_\mu(Y_s) -b_\mu(\mil_{n,s})) \, dW_s \\
		& \qquad  \qquad + \int_0^t \bigl(b_\mu(\mil_{n,s})- b_\mu(\mil_{n,\usn}) 
		- b_\mu'(\mil_{n,\usn})(\mil_{n, s}-\mil_{n,\usn})  \bigr) \, dW_s \\
		& \qquad\qquad + \int_0^t \Bigl( b_\mu(b_\mu')^2(\mil_{n,\usn} ) \int_{\usn}^s (W_u-W_{\usn} ) \, dW_u \Bigr) \, dW_s. 
	\end{aligned}
\end{equation}

Using  the Burkholder-Davis-Gundy  inequality, the Lipschitz continuity of $b_\mu$, see Lemma \ref{lem2NEW}, and the H\"older inequality we obtain that there exist $c_1,c_2\in (0,\infty)$ such that for  all $t\in [0,1]$ and all $n\in \N$, 
\begin{equation}\label{mil2}
	\begin{aligned}
	\EE \biggl[ \,\sup_{s\in [0,t] }  \biggl| \int_0^s (b_\mu(Y_s) -b_\mu(\mil_{n,s})) \, dW_s  \biggr|^p \biggr] 
	&  \le c_1 \EE \biggl[ \, \biggl| \int_0^t |b_\mu(Y_s) -b_\mu(\mil_{n,s})|^2 \, ds  \biggr|^{p/2} \biggr] \\
	&  \le c_2  \EE \biggl[ \,  \int_0^t |Y_s -\mil_{n,s}|^p \, ds  \biggr] \\
	& \le c_2   \int_0^t \EE \biggl[\sup_{u\in [0,s] }  |Y_u -\mil_{n,u}|^p\biggr ] \, ds.
	\end{aligned}
\end{equation}	

By Lemma \ref{lem2NEW}, $b_\mu$ is absolutely continuous with weak derivative 
$b_\mu'=-2\mu\circ G_\mu^{-1}$.
Using the latter fact as well as $\mu \in C^\alpha(\R)$ and the Lipschitz continuity of $G_\mu^{-1}$, see Lemma \ref{lem1NEW}(ii), we obtain that there 
exist $c_1, c_2\in (0,\infty)$ 
such that for all $x,y\in \R$ with $x\le y$,
\begin{align*}
		|	b_\mu (y)- b_\mu(x) - b_\mu'(x)(y-x)|& = \biggl | \int_x^y (b_\mu'(u)-b_\mu'(x))\, du \biggr| 
\leq c_1 \int_x^y |G_\mu^{-1}(u)-G_\mu^{-1}(x)|^\alpha\, du\\
&\le c_2 \int_x^y (u-x)^\alpha \, du \le c_2 (y-x)^{1+\alpha}.
\end{align*}	
Using the latter estimate, the Burkholder-Davis-Gundy inequality and the H\"older inequality we conclude that there exist $c_1,c_2\in (0,\infty)$ such that for
%T wieder weg:  all $t\in[0,1]$  and 
all  $n\in \N$, 
\begin{equation}\label{mil4}
	\begin{aligned}
		&	\EE \biggl[ \,
		%T rückgängig, zu umständlich \sup_{s\in [0,t] }
		\sup_{s\in [0,1]}
		  \biggl| \int_0^s \bigl(b_\mu(\mil_{n,s})- b_\mu(\mil_{n,\usn}) -  b_\mu'(\mil_{n,\usn})(\mil_{n,s}-\mil_{n,\usn})  \bigr) \, dW_s  \biggr|^p \biggr] \\
		& \qquad \qquad \le c_1 \EE \biggl[ \, \biggl| 
		%T \int_0^t  
		 \int_0^1  
		\bigl|b_\mu(\mil_{n,s})- b_\mu(\mil_{n,\usn}) -  b_\mu'(\mil_{n,\usn})(\mil_{n,s}-\mil_{n,\usn})  \bigr|^2 \, ds  \biggr|^{p/2} \biggr] \\
		& \qquad \qquad \le c_2  
		%T  \int_0^t \EE \bigl[ 
		 \int_0^1  \EE \bigl[ 
		|\mil_{n,s} -\mil_{n,\usn}|^{p(1+\alpha)}\bigr] \, ds  . 
	\end{aligned}
\end{equation}	
Since $\mu$ is bounded and $G_\mu'$ is bounded, see Lemma \ref{lem1NEW}(i), we get that $b_\mu$ and $b_\mu'$ are bounded as well. Hence there exist $c_1,c_2\in (0,\infty)$ such that for  all $s\in [0,1]$ and all $n\in \N$, 
\begin{equation}\label{mil5}
	\begin{aligned}
		&	\EE \bigl[ |\mil_{n,s} -\mil_{n,\usn}|^{p(1+\alpha)}   \bigr] \\
		& \qquad = \EE \bigl[ \bigl| b_{\mu}(\mil_{n, \usn})(W_{s} - W_{\usn}) + \frac{1}{2} b_{\mu} b_{\mu}'(\mil_{n, \usn})((W_{s} - W_{\usn})^2 - (s-\usn)) \bigr|^{p(1+\alpha)}   \bigr] \\
		& \qquad \le c_1 \EE \bigl[ \bigl( |W_{s} - W_{\usn}|+(W_{s} - W_{\usn})^2 + 1/n \bigr)^{p(1+\alpha)}  \bigr] \\
			& \qquad \le \frac{c_2}{ n^{p(1+\alpha)/2}}.
	\end{aligned}
\end{equation}	
Combining \eqref{mil4} and \eqref{mil5}, we conclude that there  exists $c\in (0,\infty)$ such that for  all  $n\in \N$, 
\begin{equation}\label{mil6}
	\EE \biggl[ \,
	%T wieder weg: \sup_{s\in [0,t] } 
	\sup_{s\in [0,1]}
	 \biggl| \int_0^s \bigl(b_\mu(\mil_{n,s})- b_\mu(\mil_{n,\usn}) -  b_\mu'(\mil_{n,\usn})(\mil_{n,s}-\mil_{n,\usn})  \bigr) \, dW_s  \biggr|^p \biggr] \le \frac{c}{ 	n^{p(1+\alpha)/2}}.
\end{equation}

Finally, by the boundedness of $b_\mu$ and $b_\mu'$, the Burkholder-Davis-Gundy inequality and the H\"older inequality, we obtain that there exist $c_1,c_2,c_3\in (0,\infty)$ such that for  all  $n\in \N$, 
\begin{equation}\label{mil7}
	\begin{aligned}
	&	\EE \biggl[ \,\sup_{t\in [0,1] }  \biggl|  \int_0^t \Bigl( b_\mu(b_\mu')^2(\mil_{n,\usn} ) \int_{\usn}^s (W_u-W_{\usn} ) \, dW_u \Bigr) \, dW_s  \biggr|^p \biggr] \\	
	& \qquad \le c_1 \EE \biggl[ \,  \biggl|  \int_0^1  b_\mu^2(b_\mu')^4(\mil_{n,\usn} ) \Bigl(\int_{\usn}^s (W_u-W_{\usn} ) \, dW_u \Bigr)^2 \, ds  \biggr|^{p/2} \biggr] \\	
	& \qquad \le c_2  \int_0^1  \EE \bigl[| (W_s-W_{\usn} )^2 - (s-\usn)|^p\bigr] \, ds \\	
	& \qquad \le \frac{c_3}{ n^{p}}.
	\end{aligned}
\end{equation}	
Combining \eqref{mil1} with \eqref{mil2}, \eqref{mil6} and \eqref{mil7} we obtain that  there  exists $c\in (0,\infty)$ such that for all $t\in [0,1]$ and all $n\in \N$,
 \begin{equation}\label{mil8}
\EE \biggl[\sup_{s\in [0,t] }  |Y_s -\mil_{n,s}|^p\biggr ]  \le   c \int_0^t \EE \biggl[\sup_{u\in [0,s] }  |Y_u -\mil_{n,u}|^p\biggr ] \, ds + \frac{c }{n^{p(1+\alpha)/2}}.
\end{equation}
Since $b_\mu$ and $b_\mu'$ are bounded, it is straightforward to see that 	$	\EE\bigl[\,\|Y\|_\infty^p\bigr] + \EE\bigl[\,\| \mil_n\|_\infty^p\bigr]< \infty$. Hence, \eqref{milest} is a consequence of \eqref{mil8} and the Gronwall inequality.
\end{proof}

\section{Preliminary estimates}\label{prelim}

%T dieser Abschnitt ist fast wortgleich mit Paper zu sobolev -- später etwas modifizieren
In the following let 
$(\Omega,\mathcal A,\PP)$
be a probability space, let $W\colon [0,1]\times \Omega\to \R$ be a standard Brownian motion on $[0,1]$, 
let $x_0\in\R$,
let $\mu\colon \R\to\R$ be measurable and bounded  and let $X$ denote the strong solution of the corresponding SDE~\eqref{sde0}. Moreover, for any continuous stochastic process $V\colon [0,1]\times \Omega\to \R$ we use 
\[
I(V) = \Bigl(I_t(V) = \int_0^t V_s\, ds\Bigr)_{ t\in [0,1] }
\]
to denote the process obtained by pathwise integrating $V$.

For every $p\in [1,\infty)$ and every discretization $\tau =\{t_1,\dots,t_n\}$ of $[0,1]$ with $0 < t_1 <\dots <t_n = 1$  we use 
\[
e_p(\tau) = \inf_{        g \colon \R^{2n} \to \R \text{ measurable}} \EE\bigl[  |X_1-g(W_{t_1},\dots,W_{t_n}, I_{t_1}(W),\dots,I_{t_n}(W) )|^p\bigr]^{1/p}
\]
to denote the smallest possible 
$L^p$-error 
that can be achieved for approximating  $X_1$ based on  $W_{t_1},\dots,W_{t_n}, I_{t_1}(W),\dots,I_{t_n}(W) $. To obtain a lower bound for $e_p(\tau)$ we generalize the coupling approach from~\cite{MGY23} and construct a Brownian motion $\tW^\tau\colon [0,1]\times \Omega\to\R$ such that almost surely $(W,I(W))$ coincides with $(\tW^\tau, I(\tW^\tau) )$ at the points $t_1,\dots,t_n$ and we analyse the  
$L^p$-distance of the solution  $\tX^\tau$ of the 
SDE~\eqref{sde0} with driving Brownian motion $\tW^\tau$ and $X$ at the final time, see Lemma~\ref{lemf1}.

Throughout the rest of the paper we put
\[
t_0 =0. 
\]
Define a continuous stochastic process $\Wcpi\colon [0,1]\times \Omega\to\R$ by
	\[
	\Wcpi_t= \EE[W_t | W_{t_1},\dots,W_{t_n},I_{t_1}(W),\dots,I_{t_n}(W)], \quad t\in [0,1],
	\]
	and note that, almost surely, for every $i\in\{1,\dots,n\}$, 
	\begin{equation}\label{uv1}
	\Wcpi_{t_i} = W_{t_i} \text{ and } \int_0^{t_i} \Wcpi_t\, dt = \EE\Bigl[ \int_0^{t_i} W_t\, dt  \Bigl| W_{t_1},\dots,W_{t_n},I_{t_1}(W),\dots,I_{t_n}(W)\Bigr] = I_{t_i}(W).
	\end{equation}

Put
	\[
	Z^\tau = W- \Wcpi.
	\]
		Without loss of generality, we may assume that $(\Omega,\mathcal A,\PP)$ 
		is rich enough to carry a continuous process $\widetilde Z^\tau\colon [0,1]\times \Omega\to\R$ such that $W$ and $\widetilde Z^\tau$ are independent and $Z^\tau \stackrel{d}{=}\widetilde Z^\tau$. We define a Brownian motion $\Wtpi\colon [0,1]\times \Omega\to\R$ by
		\[
		\Wtpi = \Wcpi   + \Ztpi.
		\]
		By \eqref{uv1} we have almost surely,  
	\begin{equation}\label{uv2a}
\forall \, i\in\{1,\dots,n\}\colon\,	\Wtpi_{t_i} = W_{t_i} \text{ and } I_{t_i}(\Wtpi) = I_{t_i}(W).
\end{equation}

We use 
\begin{equation}\label{extra1}
	\tX^\tau = (\tX_t^\tau )_{t\in[0,1]}
\end{equation}
to denote the strong solution of the SDE~\eqref{sde0} with driving Brownian motion $\tW^\tau$ in place of $W$
and 
\[
\Fc^{W,\tW^\tau} = \bigl(\Fc^{W,\tW^\tau}_t = \sigma(\{(W_s,\tW^\tau_s)| s\in [0,t]\})\bigr)_{t\in [0,1]}
\]
to denote the filtration generated by the process $(W,\tW^\tau)$. 

For later purposes we next show that, for every $i\in\{1,\dots,n\}$, the processes  $(W_t,\tW^\tau_t)_{t\le t_i}$ and $(W_t-W_{t_i},\tW^\tau_t-\Wtpi_{t_i})_{t\in [t_i,1]}$ are independent.
Let
$\overline W^\tau\colon [0,1]\times \Omega\to\R$ denote the piecewise linear interpolation of $W$ on $[0,1]$ at the points $t_0, \ldots, t_{n}$, i.e. 
\[
\overline W^\tau_t=\tfrac{t-t_{i-1}}{t_i-t_{i-1}}\,W_{t_i}+\tfrac{t_i-t}{t_i-t_{i-1}}\, W_{t_{i-1}}, \quad t\in [t_{i-1}, t_i],
\]
for $i\in\{1, \ldots, n\}$, and  put
\[
B^\tau=W-\overline W^\tau.
\]
It is well known that $(B^\tau_t)_{t\in [t_{i-1}, t_i]}$ is a Brownian bridge on $[t_{i-1}, t_i]$ for every $i\in\{1, \ldots, n\}$ and that the processes $(B^\tau_t)_{t\in [t_{0}, t_1]}, \ldots,(B^\tau_t)_{t\in [t_{n-1}, t_{n}]}, \overline W^\tau$ are independent. Hence, 
	\begin{equation}\label{uv01}
W_{t_1}-W_{t_0},\ldots, W_{t_n}-W_{t_{n-1}}, (B^\tau_t)_{t\in [t_{0}, t_1]}, \ldots,(B^\tau_t)_{t\in [t_{n-1}, t_{n}]}, \Ztpi \text{  are independent. }		
		\end{equation}
Moreover,
\[
\sigma\bigl(\{ W_{t_1},\dots,W_{t_n},I_{t_1}(W),\dots,I_{t_n}(W)   \}\bigr) = \sigma\Bigl(\Bigl\{ W_{t_1},\dots,W_{t_n},\int_{0}^{t_1} B^\tau_s\,ds,\dots, \int_{t_{n-1}}^{1} B^\tau_s\,ds \Bigr\}\Bigr), 
\]
and therefore, for every  $i\in\{1, \ldots, n\}$ and every $t\in[t_{i-1},t_i]$,
\begin{equation}\label{uv2}
	\begin{aligned}
\Wcpi_t	& = \EE\Bigl[ \overline W^\tau_t + B^\tau_t \Bigl| W_{t_1},\dots,W_{t_n},\int_{0}^{t_1} B^\tau_s\,ds,\dots, \int_{t_{n-1}}^{1} B^\tau_s\,ds\Bigr]\\
& = \overline W^\tau_t +  \EE\Bigl[  B^\tau_t \Bigl|  \int_{t_{i-1}}^{t_i} B^\tau_s\,ds\Bigr],
\end{aligned}
	\end{equation}
which yields that for every  $i\in\{1, \ldots, n\}$ and every $t\in[t_{i-1},t_i]$,
\begin{equation}\label{uv3}
	\begin{aligned}
	Z^\tau_t & = B^\tau_t - \EE\Bigl[  B^\tau_t \Bigl|  \int_{t_{i-1}}^{t_i} B^\tau_s\,ds\Bigr] \\
	& = B^\tau_t - \EE\Bigl( B^\tau_t  \int_{t_{i-1}}^{t_i}  B^\tau_s\,ds\Bigr)\, \Bigl( \EE \Bigl[\Bigl( \int_{t_{i-1}}^{t_i} B^\tau_s\,ds\Bigr)^2\Bigr] \Bigr)^{-1} 
	 \, \int_{t_{i-1}}^{t_i}  B^\tau_s\,ds\\
	 & =B^\tau_t- \frac{6(t_i-t)(t-t_{i-1})}{(t_i-t_{i-1})^3} \int_{t_{i-1}}^{t_i} B^\tau_s\, ds.
	\end{aligned}
\end{equation}
It follows that 
\begin{equation}\label{uv4}
	\begin{aligned}
& \overline W^\tau, \Bigl(\int_{t_0}^t B^\tau_t\, dt \Bigr)_{t\in [t_{0}, t_1]},  (Z^\tau_t)_{t\in [t_{0}, t_1]},\ldots,\Bigl(\int_{t_{n-1}}^{t_n}B^\tau_t\, dt\Bigr)_{t\in [t_{n-1}, t_{n}]},(Z^\tau_t)_{t\in [t_{n-1}, t_{n}]} \\
& \text{are independent.}
\end{aligned}
\end{equation}
Observing ~\eqref{uv2} and~\eqref{uv3} it is easy to see that, for every $i\in\{1,\dots,n\}$,
\begin{align*}
(W_t)_{t\in [0,t_i]} & = \psi_1\Bigl(W_{t_1}-W_{t_0}, \ldots, W_{t_i}-W_{t_{i-1}},(B^\tau_t)_{t\in [0, t_i]} \Bigr),\\
(\Wtpi_t)_{t\in [0, t_i]} & = \psi_2\Bigl( W_{t_1}-W_{t_0}, \ldots, W_{t_i}-W_{t_{i-1}}, (B^\tau_t)_{t\in [0, t_i]}, (\Ztpi_t)_{t\in [0,t_i]}\Bigr),\\
(W_t- W_{t_i})_{t\in[t_i,1]} & = \psi_3\Bigl( W_{t_{i+1}}-W_{t_i}, \ldots, W_{t_n}-W_{t_{n-1}}, (B^\tau_t)_{t\in [t_{i}, 1]}\Bigr),\\
(\Wtpi_t- \Wtpi_{t_i})_{t\in[t_i,1]} & = \psi_4\Bigl( W_{t_{i+1}}-W_{t_i}, \ldots, W_{t_n}-W_{t_{n-1}}, (B^\tau_t)_{t\in [t_{i}, 1]}, (\Ztpi_t)_{t\in [t_i,1]}\Bigr)
\end{align*}
with appropriate measurable mappings $\psi_1\colon\R^i\times C([0,t_i];\R)\to C([0,t_i];\R)$, $\psi_2\colon\R^i\times C([0,t_i];\R)\times C([0,t_i];\R)\to C([0,t_i];\R)$, $\psi_3\colon\R^{n-i}\times C([t_i,1];\R)\to C([t_i,1];\R)$, $\psi_4\colon\R^{n-i}\times C([t_i,1];\R)\times C([t_i,1];\R)\to C([t_i,1];\R)$. Hence, by~\eqref{uv01} and~\eqref{uv4},

\begin{equation}\label{qx2}
	\forall i\in\{1,\dots,n\}\colon \, \Fc^{W,\tW^\tau}_{t_i} \text{ and } \sigma\bigl(\{(W_t-W_{t_i},\tW^\tau_t-\tW^\tau_{t_i})\,|\, t\in[t_i,1]\}\bigr)\text{ are independent}.
\end{equation}

For all $n\in\N$ we put
\[
\mathcal T_n = \bigl\{\{t_1,\dots,t_n\}\,|\, 0<t_1< \dots <t_n=1\bigr\}
\]
and we define
\[
\mathcal T = \bigcup_{n\in\N} \mathcal T_n.
\]

\begin{lemma}\label{lemf1} Let $\mu\colon\R\to \R$ be measurable and bounded. Then, for 
	every $p\in[1, \infty)$ and
	every $\tau\in\mathcal T$,
	\[
	e_p(\tau) \geq \frac{1}{2}\,   \EE[|X_1-\widetilde X^\tau_1|^p]^{1/p}.
	\]
\end{lemma}

\begin{proof}
 Let $n\in\N$, $\tau\in\mathcal T_n$ and let $g \colon \R^{2n} \to \R$ be a measurable function. Since $\mu$ is measurable and bounded, 
	strong existence and pathwise uniqueness hold for the SDE \eqref{sde0}, see \cite{V80}.  Hence, for every $i\in\{1,\dots,n\}$, there exists a measurable function $\Psi_i\colon  C([0,t_i];\R)\times C([0,t_i];\R) \to \R$ such that, almost surely,
	\begin{equation}\label{uv5}
	X_{t_i} =  \Psi_i \bigl((\Wcpi_{t})_{t\in[0,t_i]}, (Z^\tau_t)_{t\in[0,t_i]}\bigr)\text{ and }	\Xtpi_{t_i} =  \Psi_i \bigl((\Wcpi_{t})_{t\in[0,t_i]}, (\Ztpi_t)_{t\in[0,t_i]}\bigr)
	\end{equation}
	see, e.g., [11, Theorem 1]. By \eqref{uv1},  \eqref{uv2} and~\eqref{uv5} we obtain, in particular, that there  exist measurable functions $\Psi\colon C([0,1];\R)\times C([0,1];\R)\to\R$ and $\varphi\colon C([0,1];\R)\to\R$ such that, almost surely,
	\[
		X_1=\Psi(\Wcpi, Z^\tau),\,\, \Xtpi_1 = \Psi(\Wcpi, \Ztpi) 
	\]
and
\[
 g(\Wtpi_{t_1}, \ldots, \Wtpi_{t_{n}},I_{t_1}(\Wtpi), \ldots, I_{t_{n}}(\Wtpi) )= g(W_{t_1}, \ldots, W_{t_{n}},I_{t_1}(W), \ldots, I_{t_{n}}(W) )=\varphi(\Wcpi).
\]

	Since $\PP^{(\Wcpi, Z^\tau)}=\PP^{\Wcpi} \times \PP^{Z^\tau} = \PP^{\Wcpi} \times \PP^{\Ztpi} =\PP^{(\Wcpi, \Ztpi)}$, see~\eqref{uv4} and \eqref{uv01}, we conclude by the triangle inequality that
	\begin{align*}
& \EE[|X_1- g(W_{t_1}, \ldots, W_{t_{n}},I_{t_1}(W), \ldots, I_{t_{n}}(W) )|^p]^{1/p} \\
& \qquad\qquad  =  \frac{1}{2}\bigl( \EE[|X_1- g(W_{t_1}, \ldots, W_{t_{n}},I_{t_1}(W), \ldots, I_{t_{n}}(W) )|^p]^{1/p} \\
 &  \qquad\qquad   \qquad\qquad  + \EE[|\Xtpi_1- g(\Wtpi_{t_1}, \ldots, \Wtpi_{t_{n}},I_{t_1}(\Wtpi), \ldots, I_{t_{n}}(\Wtpi) )|^p]^{1/p} \bigr) \\
 & \qquad\qquad \ge \frac{1}{2} \EE[|X_1- \Xtpi_1|^p]^{1/p}, 
	\end{align*}
which finishes the proof of the lemma. 
\end{proof}

Next, we use the transformation $G_\mu$, see~\eqref{trans}, to switch from a solution of the SDE~\eqref{sde0} to a  solution of the SDE~\eqref{sde1NEW}, see Lemma~\ref{lem2NEW}.

Let
\begin{equation}\label{extra2}
	Y=(G_\mu(X_t))_{t\in[0,1]}
\end{equation}
and for every $\tau\in \mathcal T$ define
\begin{equation}\label{extra3}
	\widetilde Y^\tau=(G_\mu(\widetilde X^\tau_t))_{t\in[0,1]}.
\end{equation}

\begin{lemma}\label{lemf2} Let $\mu\colon\R\to \R$ be measurable and bounded with $\sup_{y \in \R} |\int_0^y \mu(z) \, dz| < \infty$. 
	Then there exists $c\in (0,\infty)$ such that for every $p\in[1,\infty)$ and every $\tau \in \mathcal T$,
	\begin{equation}\label{gL2}
		\EE\bigl[|X_1-\widetilde X^\tau_1|^p\bigr]^{1/p} \geq c\,   \EE[|Y_1-\widetilde Y^\tau_1|^p]^{1/p}.
	\end{equation}
\end{lemma}

\begin{proof}
	By Lemma~\ref{lem1NEW}(i), the transformation $G_\mu$ is Lipschitz continuous, which obviously implies the claimed estimate.
\end{proof}

\framebox{text}

\begin{lemma}\label{lem3NEWa} 
	Let $\alpha\in (0,1)$
	and let $\mu \in C^\alpha(\R)$ be bounded  with $\sup_{t \in \R} |\int_0^t \mu(z) \, dz| < \infty$. Then, for every $p\in (0,\infty)$ there exists $c\in (0,\infty)$ such that 
	for all $n\in\N$ and all $\tau\in\mathcal T$ with $\{\ell/n\,|\, \ell\in\{1,\dots,n\}\} \subset \tau$, 
\begin{equation}\label{rst0000}
\max_{t\in \tau}	\EE\bigl[	|X_{t}-\tX^\tau_{t}|^p\bigr] \le \frac{c}{n^{p(1+\alpha)/2}}.
\end{equation}
\end{lemma}
\begin{proof}
		Let $n,m\in\N$ with $n\le m$ and let $\tau\in \mathcal T_m$ with $ \{\ell/n\,|\,\ell\in\{ 1,\dots,n\}\}\subset \tau$.
Let $Y^M_{n}$ and $\tY^{\tau,M}_{n}$ denote the continuous-time Milstein-type schemes with $n$ equidistant steps for $Y$ and $\tY^\tau$, respectively, see~\eqref{mil}.
Observe that by~\eqref{uv2a} we have $W_t = \Wtpi_t$ for all $t\in\tau$ and, as a consequence, $Y^M_{n,t} = \tY^{\tau,M}_{n,t}$ for every 
$t\in \{\ell/n\,|\, \ell\in\{0,\dots, n\}\}$. 
For $t\in \tau$ put $\utn=\lfloor t n\rfloor/n$. Then, for every $t\in \tau$,
\begin{align*}
	X_{t}-\tX^\tau_{t} & = (X_{\utn} - \tX^\tau_{\utn}) +	(X_{t} - X_{\utn } ) - (\tX^\tau_{t} - \tX^\tau_{\utn } ) \\
	& =(G_\mu^{-1}(Y_{\utn}) -G_\mu^{-1}(Y^M_{n,\utn}))  + (G_\mu^{-1}(\tY^{\tau,M}_{n,\utn})-G_\mu^{-1}(\tY^\tau_{\utn})) + \int_{\utn}^{t} ( \mu(X_s) - \mu(\tX^\tau_s))\, ds,
\end{align*}
which implies
\[
	|X_{t}-\tX^\tau_{t}| \le \|(G_\mu^{-1})'\|_\infty\, \bigl( \|Y -Y^M_{n}\|_\infty  + \|\tY^\tau-\tY^{\tau,M}_{n}\|_\infty\bigr) + \frac{2\|\mu\|_\infty}{n}.
\]
By Lemma~\ref{lem1NEW} and the boundedness of $\mu$  we have $\|(G_\mu^{-1})'\|_\infty + 2\|\mu\|_\infty< \infty$. Now, apply Proposition~\ref{prop1NEW} to complete the proof. 
	\end{proof}

For every $n\in\N$ we define 
\[
\tPi_n = \bigl\{   \{t_1,\dots, t_{5n}\}\,|\, 0<t_1 <\dots < t_{5n}=1, \{j/(4n)\,|\, j\in\{1,\dots, 4n\}\} \subset \{t_1,\dots,t_{5n} \}\bigr\}.
\]
Clearly, every $\tau\in \tPi_n$ satisfies
\begin{equation}\label{ndisc1}
	\max_{i\in \{1,\dots,5n\}} (t_i -t_{i-1}) \le 1/(4n).
\end{equation}
Moreover, it is easy to check that 
\begin{equation}\label{ndisc2}
	\forall \tau \in \mathcal T_n \, \exists \tpi\in \tPi_n\colon \, \tau \subset \tpi	
\end{equation}
and 
\begin{equation}\label{ndisc3}
	\forall \tau \in \tPi_n\colon \, \#\bigl\{ i\in \{2,\dots, 5n\}\,|\, t_{i-1}\ge 1/2\text{ and } 
	t_i-t_{i-1} = 1/(4n)
	\} \ge n.
\end{equation}

\begin{lemma}\label{lemf3NEW} Let $\alpha\in (0,1)$
	 and let $\mu \in C^\alpha(\R)$ be bounded  with $\sup_{t \in \R} |\int_0^t \mu(z) \, dz| < \infty$.
	Then there exist  $c_1,c_2\in (0,\infty)$ 
such that for all $n\in\N$, all $\tau=\{t_1,\dots,t_{5n}\}\in\tPi_n$ with $0<t_1<\dots <t_{5n}=1$ and all $ i\in \{1,\dots,5n\}$, 
\begin{equation}\label{iter1}
	\begin{aligned}
		\EE\bigl[|Y_{t_i} - \widetilde Y^\tau_{t_i}|^2\bigr] & \ge \left(1-\frac{c_1}{n}\right)\, \EE\bigl[|Y_{t_{i-1}} - \widetilde Y^\tau_{t_{i-1}}|^2\bigr] \\
		& \qquad\qquad + c_2\, \EE\bigl[|(X_{t_i}- X_{t_{i-1}}) - (\tX^\tau_{t_i}- \tX^\tau_{t_{i-1}})|^2\bigr] - \frac{c_1}{n^{2+2\alpha}}.
	\end{aligned}
\end{equation}
\end{lemma}

\begin{proof}
	The proof of Lemma~\ref{lemf3NEW} is similar to the proof 
	%T zurückkorrigiert. 
	of Lemma 7 in \cite{EMGY24}.
%T sprachlich seltsam	\cite[Lemma 7]{EMGY24}.
Let $n\in\N$ and  $\tau=\{t_1,\dots,t_{5n}\}\in\tPi_n$ with $0<t_1<\dots <t_{5n}=1$ and
put 
\[
\Delta_i  = \EE\bigl[|Y_{t_i} - \widetilde Y^\tau_{t_i}|^2\bigr]^{1/2}
\]
for  $i\in\{0,\dots,5n\}$ and 
\[
\alpha_i = (Y_{t_i}- \tY^\tau_{t_i}) - (Y_{t_{i-1}}- \tY^\tau_{t_{i-1}})
\]
for  $i\in\{1,\dots,5n\}$.  

%T part of old proof from sobolev paper with mods-- from here
Fix $i\in\{1,\dots,5n\}$. 
Throughout the following we use $c\in (0,\infty)$ to denote a positive constant that does not depend on $n$ or $\tau$ or $i$. The value of $c$ may change from 
line to line.
By Lemma~\ref{lem1NEW}(i) and Lemma~\ref{lem2NEW}, the stochastic process $(b_\mu(Y_t))_{t\in[0,1]}$ is  bounded, measurable and adapted to 
$\Fc=\bigl(\Fc_t = \sigma\bigl( \sigma(\{W_s \,|\, s\in [0,t]\})\cup \mathcal N(\PP))\bigr)
\bigr)_{t\in [0,1]}$, where $\mathcal N(\PP)$ is the collection of $\PP$-null sets.
Hence, 	for all $r,t\in[0,1]$,
\begin{equation}\label{rstuv1}
	\EE[\tY_r^\tau \tY^\tau_t]	=	\EE[Y_r Y_t]  = \EE\biggl[\int_0^{\min(r,t)} b_\mu^2(Y_u)\, du\biggr] + G^2_\mu(x_0).
\end{equation}	
By the Lipschitz continuity of $b_\mu$,  strong existence and pathwise uniqueness hold for the
SDE \eqref{sde1NEW}. Hence, there exists a measurable function $g\colon (C([0,t_{i-1}];\R))^2\to \R$  such that $\PP$-almost surely, 
\[
\tY_{t_{i-1}}^\tau = g((W_u)_{u\in [0,t_{i-1}]}, ( \Ztpi_u)_{u\in [0,t_{i-1}]}),
\]
see, e.g., \cite[Theorem 1]{Ka96}.
Since $W$ and $\Ztpi$ are independent and $(W_u)_{u\in [0,t_{i-1}]}$ is measurable with respect to $\Fc_{t_{i-1}}$,
we conclude that
\begin{equation}\label{rstuv2}
	\begin{aligned}
		& \EE[\tY_{t_{i-1}}^\tau  (	Y_{t_i}-Y_{t_{i-1}})] \\
		& \qquad\qquad = \EE\Bigl[g\bigl((W_u)_{u\in [0,t_{i-1}]}, (\Ztpi_u)_{u\in [0,t_{i-1}]}\bigr) \int_{t_{i-1}}^{t_i} b_\mu(Y_s)\, dW_s\Bigr]\\
		& \qquad\qquad = \int_{C([0,t_{i-1}];\R)} \EE\Bigl[g\bigl((W_u)_{u\in [0,t_{i-1}]}, f\bigr) \int_{t_{i-1}}^{t_i} b_\mu(Y_s)\, dW_s\Bigr]\, \PP^{(\Ztpi_u)_{u\in [0,t_{i-1}]} }(df)\\
		& \qquad\qquad = \int_{C([0,t_{i-1}];\R)} \EE\Bigl[ \int_{t_{i-1}}^{t_i} g\bigl((W_u)_{u\in [0,t_{i-1}]}, f\bigr) b_\mu(Y_s)\, dW_s\Bigr]\, \PP^{(\Ztpi_u)_{u\in [0,t_{i-1}]} }(df).
	\end{aligned} 
\end{equation}
By the boundedness of $ b_\mu$  and \eqref{rstuv1},
\[
\EE\bigl[\sup_{s\in [t_{i-1},t_i]}| \tY_{t_{i-1}}^\tau b_\mu(Y_s)|^2 \bigr] \le \|b_\mu\|^2_\infty \EE \bigl[|\tY_{t_{i-1}}^\tau|^2\bigr] \le 2\|b_\mu\|^2_\infty ( |G_\mu(x_0)|^2 + \|b_\mu\|^2_\infty) <\infty, 
\]
which implies that for $\PP^{(\Ztpi_u)_{u\in [0,t_{i-1}]} }$-almost all $f\in C([0,t_{i-1}];\R)$,
\[
\EE\bigl[\sup_{s\in [t_{i-1},t_i]}|g\bigl((W_u)_{u\in [0,t_{i-1}]}, f\bigr) b_\mu(Y_s)|^2 \bigr]  <\infty.
\]
Hence, for $\PP^{(\Ztpi_u)_{u\in [0,t_{i-1}]} }$-almost all $f\in C([0,t_{i-1}];\R)$,
\[
\EE\Bigl[ \int_{t_{i-1}}^{t_i} g\bigl((W_u)_{u\in [0,t_{i-1}]}, f\bigr) b_\mu(Y_s)\, dW_s\Bigr] =0,
\]
which jointly with~\eqref{rstuv2} yields
\begin{equation}\label{rstuv3}
	\EE[\tY_{t_{i-1}}^\tau  (	Y_{t_i}-Y_{t_{i-1}})] =0.
\end{equation}
For reasons of symmetry, we also have 
\begin{equation}\label{rstuv4}
	\EE[Y_{t_{i-1}}  (	\tY^\tau_{t_i}-\tY^\tau_{t_{i-1}})] =0.
\end{equation}

By~\eqref{rstuv1}, ~\eqref{rstuv3} and ~\eqref{rstuv4} we obtain that for all $U\in \{Y_{t_{i-1}}, \tY^\tau_{t_{i-1}}\}$ and $V\in \{Y_{t_i}-Y_{t_{i-1}}, \tY^\tau_{t_i}-\tY^\tau_{t_{i-1}}\}$,
\[
\EE[ UV] = 0.
\]
As a consequence we get $\EE[ (Y_{t_{i-1}}- \tY^\tau_{t_{i-1}})\, \alpha_i] =0$, and therefore 
\begin{equation}\label{q1}
	\Delta_i^2 = \Delta_{i-1}^2 + 2\EE[ (Y_{t_{i-1}}- \tY^\tau_{t_{i-1}})\, \alpha_i] + \EE[\alpha_i^2] =  \Delta_{i-1}^2 + \EE[\alpha_i^2].
\end{equation}

By the smoothness properties of the function $G_\mu$, see Lemma~\ref{lem1NEW}, we derive
\begin{align*}
	\alpha_i & = (G_\mu(X_{t_i}) - G_\mu(\tX^\tau_{t_i})) -  (G_\mu(X_{t_{i-1}}) - G_\mu(\tX^\tau_{t_{i-1}})) \\
	& = \int_{\tX^\tau_{t_i}}^{X_{t_i}} G_\mu'(u)\, du -  \int_{\tX^\tau_{t_{i-1}}}^{X_{t_{i-1}}} G_\mu'(u)\, du \\
	& = \beta_i + \gamma_i +\delta_i,
\end{align*}
where
\begin{align*}
	\beta_i & = \int_{\tX^\tau_{t_i}}^{X_{t_i}} (G_\mu'(u) - G_\mu'(\tX^\tau_{t_i}))\, du -  \int_{\tX^\tau_{t_{i-1}}}^{X_{t_{i-1}}} (G_\mu'(u)-G_\mu'(\tX^\tau_{t_{i-1}}))\, du,\\
	\gamma_i & = \bigl(G_\mu'(\tX^\tau_{t_i})-G_\mu'(\tX^\tau_{t_{i-1}})\bigr)\, (X_{t_{i-1}} - \tX^\tau_{t_{i-1}}),\\
	\delta_i & = G_\mu'(\tX^\tau_{t_i})\, \bigl( (X_{t_i} - \tX^\tau_{t_i}) - (X_{t_{i-1}} - \tX^\tau_{t_{i-1}})\bigr).
\end{align*}
Hence
\begin{equation}\label{q2}
	\EE\bigl[\alpha_i^2\bigr] \ge \frac{1}{2}\EE\bigl[\delta_i^2\bigr] - \EE\bigl[(\beta_i + \gamma_i)^2\bigr]
	\ge \frac{1}{2}\EE\bigl[\delta_i^2\bigr] - 2\EE\bigl[\beta_i^2\bigr] -2\EE\bigl [\gamma_i^2\bigr].
\end{equation}

We proceed by estimating $\EE[\beta_i^2]$. Employing again Lemma~\ref{lem1NEW} we get
\begin{equation}\label{q3a}
	\begin{aligned}
		\beta_i^2& = \biggl| -2\int_{\tX^\tau_{t_i}}^{X_{t_i}} \int_{\tX^\tau_{t_i}}^u\mu(v)G_\mu'(v)\, dv\, du +2  \int_{\tX^\tau_{t_{i-1}}}^{X_{t_{i-1}}} \int_{\tX^\tau_{t_{i-1}}}^u \mu(v)G_\mu'(v)\, dv\, du\biggr|^2\\
		& \le \|\mu\|_\infty^2\, \|G_\mu'\|_\infty^2 \,\bigl( |X_{t_i}-\tX^\tau_{t_i}  |^2 +  |X_{t_{i-1}}-\tX^\tau_{t_{i-1}}|^2\bigr)^2.
	\end{aligned}
\end{equation} 

Hence, by Lemma \ref{lem3NEWa} with $p=4$, we obtain
\begin{equation}\label{q4a}
	\EE\bigl[\beta_i^2\bigr] \le \frac{c}{n^{2+2\alpha}}.
\end{equation}

Next, we provide an upper bound for $\EE\bigl[\gamma_i^2\bigr]$. Clearly, $X_{t_{i-1}}$ and $\tX^\tau_{t_{i-1}}$ are measurable with respect to $\Fc^{W,\tW^\tau}_{t_{i-1}}$, and therefore,
\begin{equation}\label{q5}
	\EE\bigl[\gamma_i^2\bigr] = \EE\bigl[(X_{t_{i-1}}-\tX^\tau_{t_{i-1}})^2\, \EE\bigl[ (G_\mu'(\tX^\tau_{t_i})-G_\mu'(\tX^\tau_{t_{i-1}}))^2\bigr|\Fc^{W,\tW^\tau}_{t_{i-1}}\bigr]\bigr].
\end{equation}
By Lemma~\ref{lem1NEW} we obtain
\begin{equation}\label{q6}
	(X_{t_{i-1}}-\tX^\tau_{t_{i-1}})^2 = (G_\mu^{-1}(Y_{t_{i-1}})-G_\mu^{-1}(\tY^\tau_{t_{i-1}}))^2 \le \|(G_\mu^{-1})'\|_\infty^2\, 	(Y_{t_{i-1}}-\tY^\tau_{t_{i-1}})^2
\end{equation}
as well as
\begin{equation}\label{q7}
	\begin{aligned}
		(G_\mu'(\tX^\tau_{t_i})-G_\mu'(\tX^\tau_{t_{i-1}}))^2 & = \biggl(\int_{\tX^\tau_{t_{i-1}}}^{\tX^\tau_{t_i}} G_\mu''(u)\, du\biggr)^2 \\
		& \le 4\|\mu\|_\infty^2 \|G_\mu'\|_\infty^2\, (\tX^\tau_{t_{i}}-\tX^\tau_{t_{i-1}})^2\\
		& = 4\|\mu\|_\infty^2 \|G_\mu'\|_\infty^2 \biggl(\int_{t_{i-1}}^{t_i} \mu (\tX^\tau_t)\, dt + W_{t_{i}}-W_{t_{i-1}}\biggr)^2\\
		& \le 8\|\mu\|_\infty^2 \|G_\mu'\|_\infty^2 \left( \frac{\|\mu\|_\infty^2}{ 16n^2} +  (W_{t_{i}}-W_{t_{i-1}})^2\right).
	\end{aligned}
\end{equation}
Using~\eqref{qx2} we derive from~\eqref{q7} that
\begin{equation}\label{q8}
	\begin{aligned}
		\EE\bigl[ (G_\mu'(\tX^\tau_{t_i})-G_\mu'(\tX^\tau_{t_{i-1}}))^2\bigr|\Fc^{W,\tW^\tau}_{t_{i-1}}\bigr] 
		& \le c \left(\frac{1}{ n^2}+ \EE[ (W_{t_{i}}-W_{t_{i-1}})^2|\Fc^{W,\tW^\tau}_{t_{i-1}}]\right)\\
		&\le c \left( \frac{1}{ n^2} + \frac{1}{4n}\right).
	\end{aligned}	
\end{equation}
Combining~\eqref{q5}, ~\eqref{q6} and~\eqref{q8} we get 
\begin{equation}\label{q9}
	\EE\bigl[\gamma_i^2\bigr] \le \frac{c}{n} \Delta_{i-1}^2.	
\end{equation}
Combining~\eqref{q2}, ~\eqref{q4a} and~\eqref{q9} yields 
\begin{equation}\label{q10}
	\EE\bigl[\alpha_i^2\bigr] \ge \frac{1}{2}\EE[\delta_i^2] - \frac{c}{n} \Delta_{i-1}^2 - \frac{c}{n^{2+2\alpha}}.
\end{equation}
Finally, combine~\eqref{q1} with \eqref{q10}, use 
\[
\EE\bigl[\delta_i^2\bigr] \ge \inf_{x\in\R} |G_\mu'(x)|^2\, \EE\bigl[|(X_{t_i}- X_{t_{i-1}}) - (\tX^\tau_{t_i}- \tX^\tau_{t_{i-1}})|^2\bigr] 
\] 
and observe that $\inf_{x\in\R} |G_\mu'(x)| >0$, see Lemma~\ref{lem1NEW}, to complete the proof of the lemma.
\end{proof}

\begin{lemma}\label{lemf4NEW}
	Let $\alpha \in (0,1)$
	and let $\mu \in C^\alpha(\R)$ be bounded  with $\sup_{t \in \R} |\int_0^t \mu(z) \, dz| < \infty$.
	Then there exists  $c\in (0,\infty)$ 
	such that for all $n\in\N$, all $\tau=\{t_1,\dots,t_{5n}\}\in\tPi_n$ with $0<t_1<\dots <t_{5n}=1$ and all $ i\in \{1,\dots,5n\}$, 
	\begin{equation}\label{iter2a}
		\begin{aligned}
			& \EE\bigl[|(X_{t_i}- X_{t_{i-1}}) - (\tXp_{t_i}- \tXp_{t_{i-1}})|^2\bigr] \\
			& \qquad \ge \frac{1}{2} \EE\Bigl[ \Bigl| \int_{t_{i-1}}^{t_i} \bigl (\mu (X_{t_{i-1} } +W_t - W_{t_{i-1} } ) -  \mu (X_{t_{i-1}} +\tWp_t - \tWp_{t_{i-1}})  \bigr)\,
			dt\Bigr|^2\Bigr] -\frac{c}{ n^{2+\alpha(1+\alpha)}}.
		\end{aligned}
	\end{equation}
\end{lemma}

\begin{proof}
	Let $n\in\N$ and  $\tau=\{t_1,\dots,t_{5n}\}\in\tPi_n$ with $0<t_1<\dots <t_{5n}=1$. 	Let $i \in \{1, \dots, n\}$.
	Throughout this proof we use $c\in (0,\infty)$ to denote a positive constant that 
	neither depends on $n$
	nor on $\tau$
	 nor on $i$. 
	The value of $c$ may vary from 
	%T modified
	occurence to occurence.
	%T instead of:	line to line.
	%T Journal version: more consistent wrt constants
	
	Observing \eqref{uv2a}, we get 
	\begin{align*}
		&(X_{t_i}- X_{t_{i-1}}) - (\tXp_{t_i}- \tXp_{t_{i-1}}) \\
		& \qquad =  \int_{t_{i-1}}^{t_i} \bigl (\mu (X_{t_{i-1} } +W_t - W_{t_{i-1} } ) -  \mu (X_{t_{i-1}} +\tWp_t - \tWp_{t_{i-1}})  \bigr)\,
		dt \\
		& \qquad \qquad + \int_{t_{i-1}}^{t_i} \bigl (\mu (X_t) -  \mu (X_{t_{i-1}} + W_t - W_{t_{i-1}})  \bigr)\,
		dt\\
		& \qquad \qquad \qquad +  \int_{t_{i-1}}^{t_i} \bigl (\mu (X_{t_{i-1} } + \tWp_t - \tWp_{t_{i-1} } ) -  \mu (\tXp_t)  \bigr)\,
		dt
	\end{align*}	
and therefore
	\begin{equation}\label{eq7_1NEW}
		\begin{aligned}
			& \EE\bigl[|(X_{t_i}- X_{t_{i-1}}) - (\tXp_{t_i}- \tXp_{t_{i-1}})|^2\bigr] \\
			& \qquad \ge \frac{1}{2} \EE\Bigl[ \Bigl| \int_{t_{i-1}}^{t_i} \bigl (\mu (X_{t_{i-1} } +W_t - W_{t_{i-1} } ) -  \mu (X_{t_{i-1}} +\tWp_t - \tWp_{t_{i-1}})  \bigr)\,
			dt\Bigr|^2\Bigr] \\
			& \qquad \qquad - 2\Bigl(\EE\Bigl[ \Bigl| \int_{t_{i-1}}^{t_i} \bigl (\mu (X_t) -  \mu (X_{t_{i-1}} + W_t - W_{t_{i-1}})  \bigr)\,
			dt\Bigr|^2\Bigr] \\
			& \qquad \qquad \qquad \qquad + \EE\Bigl[ \Bigl| \int_{t_{i-1}}^{t_i} \bigl (\mu (X_{t_{i-1} } + \tWp_t - \tWp_{t_{i-1} } ) -  \mu (\tXp_t)  \bigr)\,
			dt\Bigr|^2\Bigr] \Bigr).
		\end{aligned}
	\end{equation}
Employing the $\alpha$-H\"older
 continuity and boundedness of $\mu$ as well as~\eqref{ndisc1} we get
	\begin{equation}\label{eq7_2NEW}
		\begin{aligned}
			&\EE\Bigl[ \Bigl| \int_{t_{i-1}}^{t_i} \bigl (\mu (X_t) -  \mu (X_{t_{i-1}} + W_t - W_{t_{i-1}})  \bigr)\,
			dt\Bigr|^2\Bigr] \\
			& \qquad \le c\, \EE\Bigl[\Bigl( \int_{t_{i-1}}^{t_i} \Bigl|\int_{t_{i-1}}^t \mu(X_u) \, du\Bigr|^{\alpha}\, dt  \Bigr)^2\Bigr]  \le c\, \Bigl(\int_{t_{i-1}}^{t_i} (t-t_{i-1})^\alpha\, dt\Bigr)^2 \le \frac{c}{ n^{2 + 2\alpha}}.
		\end{aligned}
	\end{equation}	
	Similarly, using the $\alpha$-H\"older continuity and boundedness of $\mu$ as well as~\eqref{ndisc1} and Lemma \ref{lem3NEWa} with $p= 2\alpha$, we obtain
	\begin{equation}\label{eq7_3NEW}
		\begin{aligned}
			&\EE\Bigl[ \Bigl| \int_{t_{i-1}}^{t_i} \bigl (\mu (X_{t_{i-1} } + \tWp_t - \tWp_{t_{i-1} } ) -  \mu (\tXp_t)  \bigr)\,
			dt\Bigr|^2\Bigr] \\
				& \qquad \le  c\, \EE\Bigl[ \Bigl(  \int_{t_{i-1}}^{t_i} \Bigl |  X_{t_{i-1}} - \tXp_{t_{i-1}} - \int_{t_{i-1}}^t \mu(\tXp_u) \, du\Bigr|^\alpha    \,
			dt\Bigr)^2\Bigr] \\
			& \qquad \le c\, \EE\Bigl[ \Bigl(  \int_{t_{i-1}}^{t_i} \bigl( |X_{t_{i-1}} - \tXp_{t_{i-1}}|^{\alpha} + (t-t_{i-1})^{\alpha}\bigr) \, dt \Bigr)^2\Bigr]\\
					& \qquad \le \frac{c}{n^2}\Bigl(\EE\bigl[|X_{t_{i-1}} - \tXp_{t_{i-1}}|^{2\alpha}\bigr] + \frac{1}{n^{2\alpha}}\Bigr) \\
					& \qquad \le \frac{c}{n^2}\Bigl(  \frac{1}{n^{(1+\alpha) \alpha}}+ \frac{1}{n^{2\alpha}}\Bigr). 
		\end{aligned}
	\end{equation}
Combining~\eqref{eq7_1NEW} with~\eqref{eq7_2NEW} and~\eqref{eq7_3NEW} and using the fact that $\alpha^2 \le \alpha$ completes the proof. 	
\end{proof}

We proceed with providing a lower bound for the 
first term on the right-hand side in \eqref{iter2a}.
In the sequel, we use
	\[
	\tau^*= \{1\}\in \mathcal T_1
	\]
	to denote the discretization that only contains the point $t_1=1$ and we put
	\begin{equation}\label{Axx}
	D(x) = \EE\Bigl[ \Bigl| \int_0^1 \bigl( \exp(-\bi x W_t) - \exp(-\bi x \Wtt_t)\bigr)\, dt \Bigr|^2 \Bigr]
\end{equation}
for $x\in\R$.

\begin{lemma}\label{NEWx}
	Let $\mu\colon\R\to \R$ 
	be measurable and  bounded and let $f\colon\R \rightarrow \R$ be measurable, $2\pi$-periodic and bounded. Then there exists $c\in (0,\infty)$ such that for all $n\in\N$, all $\tau=\{t_1,\dots,t_{5n}\}\in\tPi_n$ with $0<t_1<\dots <t_{5n}=1$ and all $ i\in \{1,\dots,5n\}$ with $t_{i-1} \ge 1/2$,
	\begin{align*}
	&	\EE\Bigl[\Bigl| \int_{t_{i-1}}^{t_i} \bigl (f (X_{t_{i-1} } +W_t - W_{t_{i-1} } ) -  f (X_{t_{i-1}} +\tWp_t - \tWp_{t_{i-1}})  \bigr)\, dt \Bigr|^2\Big] \\
	& \qquad\qquad  \qquad\qquad \ge c (t_i - t_{i-1})^{2} \sum_{j \in \Z} |\fh_j|^2 D\bigl(j \sqrt{t_i - t_{i-1}} \bigr).
	\end{align*}
\end{lemma}

\begin{proof}
	Let $n\in\N$, let $\tau=\{t_1,\dots,t_{5n}\}\in\tPi_n$ with $0<t_1<\dots <t_{5n}=1$  and let $i \in \{1, \dots, 5n\}$ with $t_{i-1} \ge 1/2$. 
	
	Since $\mu$ is
	measurable and
	 bounded, we may apply~\cite[Theorem 1]{QZ02} to obtain that, for every $t\in (0,1]$, 
	the distribution of $X_{t}$ has a Lebesgue density $p_{t}$ and that there exists  $c\in (0,\infty)$ such that
	\[
\inf_{t\in [1/2,1]} \inf_{x\in [0,2\pi]}	p_{t}(x) \ge c.
	\]
Hence, by~\eqref{qx2} 
	\begin{equation}\label{rsv1NEW}
		\begin{aligned}
			&	 \EE\Bigl[\Bigl| \int_{t_{i-1}}^{t_i} \bigl (f (X_{t_{i-1} } +W_t - W_{t_{i-1} } ) -  f (X_{t_{i-1}} +\widetilde W^\tau_t - \widetilde W^\tau_{t_{i-1}})  \bigr)\, dt \Bigr|^2\Big] \\
			& \qquad\qquad = \EE\Bigl[\int_{\R}  \Bigl| \int_{t_{i-1}}^{t_i} \bigl (f (x +W_t - W_{t_{i-1} } ) -  f (x +\tWp_t - \tWp_{t_{i-1}})  \bigr)\, dt \Bigr|^2	p_{t_{i-1}}(x)\, dx \Big]\\
			& \qquad \qquad \ge c\, \EE\Bigl[\int_{0}^{2\pi} \Bigl| \int_{t_{i-1}}^{t_i} \bigl (f (x +W_t - W_{t_{i-1} } ) -  f (x +\tWp_t - \tWp_{t_{i-1}})  \bigr)\, dt \Bigr|^2dx\Big].
		\end{aligned}
	\end{equation}

Note that 	
\begin{equation}\label{uv111a}
 (\overline W^{\tau}_t-W_{ t_{i-1} } )_{ t\in [t_{i-1},t_i] } \stackrel{d}{=}\Bigl( \frac{ t-t_{i-1}}{ \sqrt{t_i-t_{i-1}}} W_1\Bigr)_{t\in [t_{i-1},t_i]} \stackrel{d}{=}\sqrt{ t_i-t_{i-1}}\, \Bigl( \overline W^{\tau*}_{\frac{t-t_{i-1}}{t_i-t_{i-1}}    }\Bigr)_{t\in [t_{i-1},t_i]} 
\end{equation}
and, employing~\eqref{uv3}, we have
\begin{equation}\label{uv111b}
	\begin{aligned}
& \Bigl( B^\tau_t, \EE\Bigl[  B^\tau_t \Bigl|  \int_{t_{i-1}}^{t_i} B^\tau_s\,ds\Bigr]\Bigr)_{t\in [t_{i-1},t_i]}  \\
& \qquad\qquad  \stackrel{d}{=} \sqrt{t_i-t_{i-1}}  \Bigl( B^{\tau*}_{ \frac{t-t_{i-1} } { t_i-t_{i-1}}}, \frac{6(t_i-t)(t-t_{i-1})}{(t_i-t_{i-1})^3}  \int_{t_{i-1}}^{t_i} 
  B^{\tau*}_{\frac{s-t_{i-1} } {t_i-t_{i-1} } }\, ds \Bigr)_{t\in [t_{i-1},t_i]}\\
& \qquad\qquad = \sqrt{t_i-t_{i-1}}   \Bigl( B^{\tau*}_{ \frac{t-t_{i-1} } { t_i-t_{i-1}}}, \EE\Bigl[ B^{\tau*}_{ \frac{t-t_{i-1} } { t_i-t_{i-1}}}\Big| \int_0^1 B^{\tau*}_s \, ds\Bigr] \Bigr)_{t\in [t_{i-1},t_i]},
  \end{aligned}
\end{equation}
which, in particular, yields
\begin{equation}\label{uv111c}
	\begin{aligned}
	(\Ztpi_t)_{t\in [t_{i-1},t_i]} 	&  \stackrel{d}{=} 	(Z^{\tau}_t)_{t\in [t_{i-1},t_i]} = \Bigl( B^\tau_t - \EE\Bigl[  B^\tau_t \Bigl|  \int_{t_{i-1}}^{t_i} B^\tau_s\,ds\Bigr]\Bigr)_{t\in [t_{i-1},t_i]} \\
	&  \stackrel{d}{=}  \sqrt{t_i-t_{i-1}}   \Bigl( B^{\tau*}_{ \frac{t-t_{i-1} } { t_i-t_{i-1}}}- \EE\Bigl[ B^{\tau*}_{ \frac{t-t_{i-1} } { t_i-t_{i-1}}}\Big| \int_0^1 B^{\tau*}_s \, ds\Bigr] \Bigr)_{t\in [t_{i-1},t_i]} \\ & =  \sqrt{t_i-t_{i-1}} 	\Bigl(Z^{\tau*}_{ \frac{t-t_{i-1} } { t_i-t_{i-1}}}\Bigr)_{t\in [t_{i-1},t_i]}   \stackrel{d}{=}  \sqrt{t_i-t_{i-1}} 	\Bigl(\widetilde Z^{\tau*}_{ \frac{t-t_{i-1} } { t_i-t_{i-1}}}\Bigr)_{t\in [t_{i-1},t_i]} 
	\end{aligned}
\end{equation}
Using~\eqref{uv111a} to \eqref{uv111c}, \eqref{uv01} and~\eqref{uv2}, we may thus conclude that
\begin{equation}\label{uv111d}
	\begin{aligned}
	&	(W_t-W_{t_{i-1}}, \Wtpi_t-\Wtpi_{t_{i-1}})_{t\in [t_{i-1},t_i]} 	\\
	&\qquad = 	\Bigl( \overline W^{\tau}_t-W_{ t_{i-1} } +    B^\tau_t ,   \overline W^{\tau}_t-W_{ t_{i-1} } +   \EE\Bigl[  B^\tau_t \Bigl |  \int_{t_{i-1}}^{t_i} B^\tau_s\,ds\Bigr] +  \Ztpi_t      \Bigr)_{t\in [t_{i-1},t_i]} \\
		& \qquad  \stackrel{d}{=}  \sqrt{t_i-t_{i-1}} 	\Bigl(\overline W^{\tau*}_{\frac{t-t_{i-1}}{t_i-t_{i-1}}}  +B^{\tau*}_{ \frac{t-t_{i-1} } { t_i-t_{i-1}}}, \\
		 & \qquad\qquad\qquad\qquad \qquad\qquad   \overline W^{\tau*}_{\frac{t-t_{i-1}}{t_i-t_{i-1}}} +\EE\Bigl[ B^{\tau*}_{ \frac{t-t_{i-1} } { t_i-t_{i-1}}}\Big| \int_0^1 B^{\tau*}_s \, ds\Bigr] + \widetilde Z^{\tau*}_{ \frac{t-t_{i-1} } { t_i-t_{i-1}}} \Bigr)_{t\in [t_{i-1},t_i]} \\
	&\qquad = 	\sqrt{t_i-t_{i-1}} 	\Bigl( W_{\frac{t-t_{i-1}}{t_i-t_{i-1}}},  \widetilde W^{\tau*}_{\frac{t-t_{i-1}}{t_i-t_{i-1}}}  \Bigr)_{t\in [t_{i-1},t_i]}.
	\end{aligned}
\end{equation}
By \eqref{uv111d} we obtain 
\begin{equation}\label{uv111e}
	\begin{aligned}
& \EE\Bigl[\int_{0}^{2\pi} \Bigl| \int_{t_{i-1}}^{t_i} \bigl (f (x +W_t - W_{t_{i-1} } ) -  f (x +\tWp_t - \tWp_{t_{i-1}})  \bigr)\, dt \Bigr|^2dx\Big]\\
& \qquad  = (t_i-t_{i-1})^2	\EE\Bigl[\int_{0}^{2\pi} \Bigl| \int_0^1 	\bigl (f (x +	\sqrt{t_i-t_{i-1}} \, W_t  ) - 
 f (x +\sqrt{t_i-t_{i-1}} \, \widetilde W^{\tau*}_t)  \bigr)\, dt \Bigr|^2dx\Big].
	\end{aligned}
\end{equation}

Define $g\colon \Omega\times \R\to \R$ by
\[
g(\omega,x) = \int_{0}^1 \bigl (f (x +\sqrt{t_i-t_{i-1}} \, W_t(\omega) ) -  f (x +\sqrt{t_i-t_{i-1}} \, \widetilde W^{\tau*}_t(\omega))  \bigr)\, dt
\]
for $(\omega,x)\in \Omega\times \R$.	
Since $f$ is measurable and bounded we obtain that for every $\omega\in\Omega$, the function $g(\omega,\cdot)\colon \R\to\R$ is measurable and bounded as well.
 Moreover,
  using the $2\pi$-periodicity of $f$, it is easy to see that for every $\omega\in\Omega$ and every $j\in\Z$, the $j$-th Fourier coefficient $\widehat g_j(\omega)$ of $g(\omega,\cdot) $ satisfies
\begin{equation}\label{cvbNEW}
\widehat g_j(\omega) = \fh_j \int_0^1 \bigl (\exp(-\bi j \sqrt{t_i-t_{i-1}} \, W_t(\omega) ) - \exp(-\bi j \sqrt{t_i-t_{i-1}} \, \widetilde W^{\tau*}_t(\omega)) \bigr) \, dt.
\end{equation}
By Parseval's theorem we obtain for every $\omega\in \Omega$, 
	\begin{equation}\label{vc23NEW}
		\begin{aligned}
			&\int_{0}^{2\pi}\Bigl[\Bigl| \int_0^1 \bigl (f (x +\sqrt{t_i-t_{i-1}} \, W_t(\omega) ) -  f (x +\sqrt{t_i-t_{i-1}} \, \widetilde W^{\tau*}_t(\omega))  \bigr)\, dt \Bigr|^2\Big] dx  \\
			& \qquad \qquad =  \sum_{j \in \Z } |\widehat g_j(\omega)|^2.
		\end{aligned}
	\end{equation}
Hence,
	\begin{equation}\label{uv111f}
	\begin{aligned}
& \EE\Bigl[\int_{0}^{2\pi} \Bigl| \int_0^1 	\bigl (f (x +	\sqrt{t_i-t_{i-1}} \, W_t  ) -  f (x +\sqrt{t_i-t_{i-1}} \, \widetilde W^{\tau*}_t)  \bigr)\, dt \Bigr|^2dx\Big] \\
&  \qquad =  \sum_{j \in \Z } |\widehat f_j|^2  \EE\Bigl[ \Bigl| \int_0^1 \bigl (\exp(-\bi j \sqrt{t_i-t_{i-1}} \, W_t ) -
 \exp(-\bi j \sqrt{t_i-t_{i-1}} \, \widetilde W^{\tau*}_t) \bigr) \, dt\Bigr|^2\Bigr].
		\end{aligned}
\end{equation}
Now, combine~\eqref{rsv1NEW}, \eqref{uv111e} and \eqref{uv111f} to complete the proof of the lemma.
\end{proof}

Next, we show that the function $D$ is bounded away from zero on the interval $[1,2]$.

\begin{lemma}\label{NEWxy} We have
	\[
	\inf_{x\in[1,2]} D(x) > 0.
	\]
\end{lemma}

\begin{proof}
	Clearly, $|D(x)| \le 4$ for all $x\in \R$. Hence, for all $x,y\in\R$,
	\begin{equation*}
		\begin{aligned}
& |D(x)-D(y)| \\	& \qquad \le 4 |\sqrt{D(x)} - \sqrt{D(y)}|  \\
& \qquad \le 4  \EE \Bigl[ \Bigl| \int_0^1 \Bigl( \bigl( \exp(-\bi x W_t) - \exp(-\bi x \Wtt_t)\bigr) - \bigl( \exp(-\bi y W_t) - \exp(-\bi y \Wtt_t)\bigr)\Bigr)\, dt \Bigr|^2 \Bigr]^{1/2}\\
& \qquad \le 4  \EE \Bigl[ \Bigl| \int_0^1 2(|W_t|+|\widetilde W^{\tau*}_t|)|x-y|  \, dt  \Bigr|^2 \Bigr]^{1/2}\\
& \qquad \le 16|x-y| \EE \Bigl[ \Bigl| \int_0^1 |W_t| \, dt  \Bigr|^2 \Bigr]^{1/2},
		\end{aligned}
	\end{equation*}
wich shows that the function $D\colon \R\to \R$ is continuous.

It remains to show that $D(x)>0$ for all $x\in [1,2]$. For  $\eps\in (0,\infty)$ we  put
\[
B_\eps = \bigl\{f\in C([0,1];\R) \,| \, \|f\|_\infty < \eps \bigr\}
\]
and we note that 
\begin{equation}\label{gauss1}
	\PP(\Wc^{\tau*}\in B_\eps) >0
	\end{equation}
since $\Wc^{\tau*}$ is a  centered Gau{\ss}ian process with continuous paths, see e.g. the proof of Lemma 5.1 in~\cite{VZ2008}.

Assume that there exists $x\in (0,\infty)$ with $D(x) = 0$ and let $\eps \in (0,\infty)$. Then, by the independence of the processes $\Wc^{\tau*}$, $Z^{\tau*}$ and  $\widetilde Z^{\tau*}$, see\eqref{uv2} and \eqref{uv4}, we obtain
\begin{align*}
0 & =  \EE\Bigl[ \Bigl| \int_0^1 \bigl( \exp(-\bi x W_t) - \exp(-\bi x \Wtt_t)\bigr)\, dt \Bigr|^2 \Bigr]\\
& = \EE\Bigl[ \Bigl| \int_0^1 \bigl( \exp(-\bi x(\Wc^{\tau*}_t +  Z^{\tau*}_t)) - \exp(-\bi x (\Wc^{\tau*}_t +  \widetilde Z^{\tau*}_t))\bigr)\, dt \Bigr|^2 \Bigr]\\
& \ge  \int_{B_\eps} \EE\Bigl[ \Bigl| \int_0^1 \bigl( \exp(-\bi x(h(t) +  Z^{\tau*}_t)) - \exp(-\bi x (h(t) +  \widetilde Z^{\tau*}_t))\bigr)\, dt \Bigr|^2 \Bigr]\, \PP^{\Wc^{\tau*}}(dh),
\end{align*} 
which jointly with~\eqref{gauss1} yields the existence of $h_\eps\in B_\eps$ such that, almost surely,
\[
\int_0^1  \exp(-\bi x(h_\eps(t) +  Z^{\tau*}_t))\, dt  = \int_0^1  \exp(-\bi x (h_\eps(t) +  \widetilde Z^{\tau*}_t))\, dt.
\]
Since the latter two integrals are independent, we conclude that there exists $c\in\C$ such that, almost surely, 
\[
\int_0^1  \exp(-\bi x(h_\eps(t) +  Z^{\tau*}_t))\, dt = c,	
	\]
which implies that there exists $c\in\R$ such that, almost surely, 	
\begin{equation}\label{kk1}
	\int_0^1  \cos(x(h_\eps(t) +  Z^{\tau*}_t))\, dt = c.
\end{equation}

For all $0 < u < v < \infty$ and $\delta\in (0,1/2) $ we put
\[
A_\delta (u,v) =  \bigl\{f\in C([0,1];\R) \,| \, \forall t\in [\delta,1-\delta]\colon u \le f(t) \le v \bigr\}
\]
and we note that 
\begin{equation}\label{gauss2}
	\PP(Z^{\tau*}\in -A_\delta (u,v) )=   	\PP(Z^{\tau*}\in A_\delta (u,v) ) >0,
\end{equation}	
see Lemma~\ref{auxNEW0}.
%T evtl noch erwähnen, dass Z symmetrische Verteilung hat?

Let 
\[
\eps= \pi/(16x),\, u_1 = 5\pi/(16x),\,  v_1 = 7\pi/(16x),\, u_2 = \pi/(32x),\, v_2 = 5\pi/(48x). 
\]
For all $h\in B_\eps$, every $\delta\in (0,1/2)$ and all $f_1\in A_\delta(u_1,v_1)$,  $f_2\in -A_\delta(u_2,v_2)$, we have, for every $t\in [\delta,1-\delta]$,
\begin{align*}
x(h(t) + f_1(t)) & \le x(\eps + v_1)= \pi/16 + 7\pi/16 = \pi/2,\\
x(h(t) + f_1 (t))  & \ge   x(-\eps + u_1) = -\pi/16 + 5\pi/16 = \pi/4,\\
x(h(t) + f_2(t)) & \le x(\eps - u_2)= \pi/16 - \pi/32 = \pi/32,\\
x(h(t) + f_2 (t))  & \ge   x(-\eps -v_2) = -\pi/16 - 5\pi/48 = -\pi/6,\\
\end{align*}
which implies that, for every $t\in [\delta,1-\delta]$,
\[
0 \le \cos(x(h(t) + f_1(t))) \le \cos(\pi/4) < \cos(\pi/6) \le  \cos(x(h(t) + f_2(t))).
\]
Hence, for  all $h \in B_\eps$ and every $\delta\in (0,1/2)$ and all $f_1\in A_\delta(u_1,v_1)$,  $f_2\in -A_\delta(u_2,v_2)$,
\begin{align*}
	\int_0^1  \cos(x(h(t) +  f_1(t)))\, dt &  \le 2\delta + 	\int_\delta^{1-\delta} \cos(x(h(t) +  f_1(t)))\, dt \le 2\delta + (1-2\delta)\cos(\pi/4),\\
		\int_0^1  \cos(x(h(t) +  f_2(t)))\, dt &  \ge -2\delta +	\int_\delta^{1-\delta} \cos(x(h(t) +  f_2(t)))\, dt \ge -2\delta + (1-2\delta)\cos(\pi/6).
\end{align*}
Since $\lim_{\delta\to 0} (2\delta + (1-2\delta)\cos(\pi/4)) = \cos(\pi/4) $ and  $\lim_{\delta\to 0} (-2\delta + (1-2\delta)\cos(\pi/6)) = \cos(\pi/6)$, there exists $\delta^*\in (0,1)$ and  $\cos(\pi/4) < \beta_1 < \beta_2 < \cos(\pi/6)$  
such that for all $h\in B_\eps$ and all $f_1\in A_{\delta^*}(u_1,v_1)$,  $f_2\in -A_{\delta^*}(u_2,v_2)$,
\begin{equation}\label{kk34}
	\int_0^1  \cos(x(h(t) +  f_1(t)))\, dt \le \beta_1 < \beta_2 \le 	\int_0^1  \cos(x(h(t) +  f_2(t)))\, dt. 
\end{equation}
Employing~\eqref{gauss2}, we conclude from~\eqref{kk34} that 
\[
\PP\Bigl(	\int_0^1  \cos(x(h_\eps(t) +  Z^{\tau*}_t))\, dt \le \beta_1 \Bigr) > 0 \text{ and } \PP\Bigl(	\int_0^1  \cos(x(h_\eps(t) +  Z^{\tau*}_t))\, dt \ge \beta_2 \Bigr) > 0, 
\]
in contradiction to~\eqref{kk1}, which completes the proof of the lemma.
	\end{proof}

\section{Proof of Theorem~\ref{thm2a}}\label{proofthm2}
We first state  properties of the Weierstrass function $\mu_\alpha$ that are crucial for the proof of Theorem~\ref{thm2a}. 

\begin{lemma}\label{mu-alpha}
	Let $\alpha\in (0,1)$. The function $\mu_\alpha$ is bounded, $2\pi$-periodic with $\int_0^{2\pi} \mu_\alpha(x)\, dx =0$ and satisfies $\mu_\alpha \in C^\alpha(\R)$.
\end{lemma}

\begin{proof} 
	It is straightforward to check that $\mu_\alpha$ is bounded, $2\pi$-periodic and $\int_0^{2\pi} \mu_\alpha(x)\, dx =0$. For the proof of $\mu_\alpha \in C^\alpha(\R)$ see e.g. ~\cite[Theorem 4.9 in Chapter II]{Zyg2002}.
\end{proof}	

Next, we provide a lower bound for the term $\EE\Bigl[ \Bigl| \int_{t_{i-1}}^{t_i} \bigl (\mu (X_{t_{i-1} } +W_t - W_{t_{i-1} } ) -  \mu (X_{t_{i-1}} +\tWp_t - \tWp_{t_{i-1}})  \bigr)\,
dt\Bigr|^2\Bigr]$ in~ \eqref{iter2a} with $\mu=\mu_\alpha$.

\begin{lemma}\label{lemf5NEW}
	Let $\alpha \in (0,1)$ and let $\mu = \mu_\alpha$. 	Then there exists $c\in (0,\infty)$ such that for all $n\in\N$, all $\tau=\{t_1,\dots,t_{5n}\}\in\tPi_n$ with $0<t_1<\dots <t_{5n}=1$ and all $ i\in \{1,\dots,5n\}$ with $t_{i-1} \ge 1/2$,
	\[
		\EE\Bigl[\Bigl| \int_{t_{i-1}}^{t_i} \bigl (\mu_\alpha (X_{t_{i-1} } +W_t - W_{t_{i-1} } ) -  \mu_\alpha (X_{t_{i-1}} +\tWp_t - \tWp_{t_{i-1}})  \bigr)\, dt \Bigr|^2\Big] \ge c\, (t_i - t_{i-1})^{2+\alpha}.
	\]
\end{lemma}

\begin{proof}
	Let $n\in\N$, let $\tau=\{t_1,\dots,t_{5n}\}\in\tPi_n$ with $0<t_1<\dots <t_{5n}=1$ and let $i \in \{1, \dots 5n\}$ with $t_{i-1} \ge 1/2$. Throughout this proof,  $c\in (0,\infty)$ denotes a  positive constant,
	which neither depends on $n$ nor on $\tau$ nor on  $i$ and may change its value from line to line.	

By Lemma \ref{mu-alpha},	
the function $\mu_\alpha$ is measurable, bounded and $2\pi$-periodic.
	Moreover, by the fact that $\sin(z) = (\exp(\bi z) - \exp(-\bi z) )/(2\bi)$ for all $z\in \R$, we obtain that for all $x \in \R$,
	\begin{equation}\label{similar}
		\begin{aligned}
			\mu_\alpha(x) = \sum_{j = 1}^\infty 2^{-\alpha j} \frac{\exp(\bi 2^j x) - \exp(-\bi 2^j x)}{2 \bi} = \sum_{j\in \Z\setminus \{0\}} \frac{\sgn(j) 2^{- \alpha |j|} }{2 \bi}\exp\bigl(\bi \sgn(j) 2^{|j|}x\bigr).
		\end{aligned} 
	\end{equation}
We may thus apply Lemma \ref{NEWx} with 
$\mu  = \mu_\alpha$ and $f = \mu_\alpha$
to obtain
	\begin{align*}
& \EE\Bigl[\Bigl| \int_{t_{i-1}}^{t_i} \bigl (\mu_\alpha (X_{t_{i-1} } +W_t - W_{t_{i-1} } ) -  \mu_\alpha (X_{t_{i-1}} +\tWp_t - \tWp_{t_{i-1}})  \bigr)\, dt 
 \Bigr|^2\Big] \\ 
 & \qquad \qquad \qquad  \ge c (t_i - t_{i-1})^{2} \sum_{j \in \Z\setminus \{0\}} \frac{2^{-2\alpha |j| }}{4} D(\sgn (j)2^{|j|}\sqrt{t_i - t_{i-1}})  \\
 & \qquad \qquad \qquad  \ge c (t_i - t_{i-1})^{2}  2^{-2\alpha j^* }  D(2^{j^*}\sqrt{t_i - t_{i-1}} ), 
\end{align*}
where
$j^\ast  =  \lceil -\log_2(\sqrt{t_i - t_{i-1}}) \rceil$. Clearly, we have $1/\sqrt{t_i-t_{i-1}} \le 2^{j^*} \le 2/\sqrt{t_i-t_{i-1}}$, and therefore, by Lemma~\ref{NEWxy},
	\begin{equation}\label{similar2}
		 2^{- 2\alpha j^\ast}  D(2^{j^*}\sqrt{t_i - t_{i-1}} )  
	 \ge c(t_i - t_{i-1})^{\alpha} 	\inf_{x\in[1,2]} D(x)  \ge c (t_i - t_{i-1})^{\alpha}, 
\end{equation}		
which finishes the proof of the lemma.
\end{proof}

We are ready to proceed with the proof of Theorem~\ref{thm2a}. 

\begin{proof}[Proof of Theorem~\ref{thm2a}]	
	
Let $\alpha\in (0,1)$. By Lemma \ref{mu-alpha} we have that $\mu_\alpha$ is bounded, $2\pi$-periodic and satisfies $\mu_\alpha \in C^\alpha(\R)$. Moreover, since $\int_0^{2\pi}\mu_\alpha(x)\, dx=0$, we obtain
\[
\sup_{y\in\R}\Bigl|\int_0^y \mu_\alpha (z) \, dz \Bigr| \le 2\pi \|\mu_\alpha\|_\infty  < \infty.
\] 
We may thus apply Lemma \ref{lemf3NEW}, Lemma \ref{lemf4NEW} and Lemma \ref{lemf5NEW} to obtain that there exist  $c_1,c_2, c_3\in (0,\infty)$ 
such that 
 for all $n\in\N$, all $\tau=\{t_1,\dots,t_{5n}\}\in\tPi_n$ with $0<t_1<\dots <t_{5n}=1$ and
 all $ i\in \{1,\dots,5n\}$ with $t_{i-1}\ge 1/2$,

\begin{equation}\label{iter2NEW}
		\EE\bigl[|Y_{t_i} - \widetilde Y^\tau_{t_i}|^2\bigr]  \ge \left(1-\frac{c_1}{n}\right)\, \EE\bigl[|Y_{t_{i-1}} - \widetilde Y^\tau_{t_{i-1}}|^2\bigr]  + c_2\,(t_i-t_{i-1})^{2+\alpha} 
		- \frac{c_3}{n^{2+\alpha(1+\alpha)}}.  
\end{equation}

Let $n\in\N$ with $n>c_1$ 
%L we have to assume n>c_1, not n\gec_1.
%T Why?
and let $\tau=\{t_1,\dots,t_{5n}\}\in\tPi_n$ with $0<t_1<\dots <t_{5n}=1$. 	
Choose the unique  $r(\tau) \in \{1, \dots, 5n\}$ with $t_{r(\tau)} = \frac{1}{2}$. 
Iteratively applying \eqref{iter2NEW} for $i=5n, \ldots, r(\tau)+1$ we obtain
	\begin{align*}
		&\EE\bigl[|Y_{1} - \tYp_{1}|^2\bigr] \ge \left(1- \frac{c_1}{n}\right)^{5n-r(\tau)}
		 \EE\bigl[|Y_{t_{r(\tau)}} - \tYp_{t_{r(\tau)}}|^2\bigr] \\
		& \qquad\qquad  \qquad\qquad+ c_2 \sum_{i=r(\tau)+1}^n\left(1- \frac{c_1}{n}\right)^{n-i} (t_i - t_{i-1})^{2 + \alpha} -
	    (5n-r(\tau))\cdot \frac{c_3}{n^{2+\alpha(1+\alpha)}}
	\end{align*}		
and hence,
using \eqref{ndisc3},
\begin{equation}\label{thm2_2NEW}
	\begin{aligned}		
		\EE\bigl[|Y_{1} - \tYp_{1}|^2\bigr]&\ge c_2 \left(1 - \frac{c_1}{n}\right)^n  \sum_{i=r(\tau)+1}^n(t_i - t_{i-1})^{2 + \alpha} 
		 -5n\cdot \frac{c_3}{n^{2+\alpha(1+\alpha)}}
		 %T added - much shorter proof
		 \\ & \ge c_2 \left(1 - \frac{c_1}{n}\right)^n \frac{1}{4^{2+\alpha} n^{1+\alpha}} 		 -5n\cdot \frac{c_3}{n^{2+\alpha(1+\alpha)}}.
	\end{aligned}
\end{equation}

Since $\lim_{n \rightarrow \infty} (1 - \frac{c_1}{n})^n = e^{-c_1}$, we obtain by~\eqref{thm2_2NEW} that there exist $c\in (0,\infty)$ and $n^\ast\in\N$ such that  for every $n\in\N$ with $n\ge n^\ast$ and all $\tau\in\tPi_n$,	
\begin{equation}\label{aux11NEW}
	\begin{aligned}
		\EE\bigl[|Y_{1} - \tYp_{1}|^2\bigr] \ge \frac{c}{n^{1+\alpha}}.
	\end{aligned}
\end{equation}
Moreover, by Lemma~\ref{lem1NEW} and Lemma~\ref{lem3NEWa} we obtain  that for every 
$p\in [1,\infty$) 
there exist $c_1,c_2\in (0,\infty)$ such that  for every $n\in\N$ and every $\tau\in \tPi_n$,
\begin{equation}\label{aux12NEW}
	\begin{aligned}
	\bigl(\EE\bigl[|Y_{1} - \tYp_{1}|^p\bigr] \bigr)^{1/p} \le   c_1 	\bigl(\EE\bigl[|X_{1} - \tXp_{1}|^p\bigr] \bigr)^{1/p} \le   \frac{c_2}{n^{(1+\alpha)/2}}.
\end{aligned}
\end{equation}
By \eqref{aux11NEW} and \eqref{aux12NEW} we may apply Lemma ~\ref{auxNEW}
in the appendix
  with  $Z=Y_{1} - \tYp_{1}$, $p=1$ and $r_1=r_2 =n^{-(1+\alpha)/2}$ and any $q\in (1,\infty)$ to obtain that   there exist $c\in (0,\infty)$ and $n^\ast\in\N$ such that  for every $n\in\N$ with $n\ge n^\ast$ and every $\tau\in \tPi_n$,
\begin{equation} \label{thm2_3NEW}
	\begin{aligned}
	 \EE\bigl[|Y_{1} - \tYp_{1}|\bigr] \ge  \frac{c}{n^{(1 + \alpha)/2}}.
	\end{aligned}
\end{equation}

Using~\eqref{ndisc2} as well as Lemma~\ref{lemf1} and Lemma~\ref{lemf2} we obtain that there exists $c\in (0,\infty)$ such that for all $n\in\N$,
\begin{equation}\label{end1}
	\begin{aligned}
		\inf_{\tau\in\mathcal T_n} e_1(\tau) & \ge 	\inf_{\tau\in\mathcal T_{\max(n,n^\ast)}}e_1(\tau) \\
		& \ge 	\inf_{\tau\in\widetilde{\mathcal T}_{\max(n,n^\ast)}}e_1(\tau) \\
		&  \ge \frac{1}{2} \inf_{\tau\in\widetilde{ \mathcal T}_{\max(n,n^\ast)}} \EE\bigl[|X_1-\tX^{\tau}_1|\bigr] \\
		& \ge c \inf_{\tau\in\widetilde{\mathcal T}_{\max(n,n^\ast)}} \EE\bigl[|Y_1-\tY^{\tau}_1|\bigr].
	\end{aligned}	
\end{equation} 
 Combining \eqref{end1}  with ~\eqref{thm2_3NEW} completes the proof of Theorem ~\ref{thm2a}.

\end{proof}

\section{Proof of Theorem~\ref{thm3}}\label{proofthm3}

We first provide properties of the function $\mus$ that are crucial for the proof of Theorem~\ref{thm3}. 

\begin{lemma}\label{lemSob1}
	Let $\alpha \in (0,1)$. Then for all $\beta \in (0, \infty)$ we have
	\begin{itemize}
		\item [(i)] $\mus$ is bounded, 
		\item [(ii)] $\mus\in C^\alpha(\R)$,
		\item [(iii)] $\mus\in \cap_{q \geq 1} L^q(\R)$.
	\end{itemize} 	
	Moreover, for all  $p \in [1,2]$ and all $\beta \in (1/p, \infty)$ we have
	\begin{itemize}	
		\item[(iv)]  $\mus \in \cap_{q \geq p}\mathsf W^{s,q}(\R)$.
	\end{itemize}
\end{lemma}

The proof of Lemma \ref{lemSob1} is shifted to the appendix. 

Next, we provide a lower bound for the term $\EE\bigl[ \bigl| \int_{t_{i-1}}^{t_i} \bigl (\mu (X_{t_{i-1} } +W_t - W_{t_{i-1} } ) -  \mu (X_{t_{i-1}} +\tWp_t - \tWp_{t_{i-1}})  \bigr)\,
dt\bigr|^2\bigr]$ in~ \eqref{iter2a} with $\mu=\mu_{\alpha,\beta}$.

\begin{lemma}\label{lemf6NEW}
	Let $\alpha \in (0,1), \beta \in (0, \infty)$ and let $\mu = \mus$. Then there exist $c_1, c_2, c_3\in (0,\infty)$ such that for all $n\in\N$, all $\tau=\{t_1,\dots,t_{5n}\}\in\tPi_n$ with $0<t_1<\dots <t_{5n}=1$ and all  $i\in \{1,\dots, 5n\}$ with $t_{i-1} \ge 1/2$,
	\begin{equation}\label{iter3}
		\begin{aligned}
			&\EE\Bigl[\Bigl| \int_{t_{i-1}}^{t_i} \bigl (\mus (X_{t_{i-1} } +W_t - W_{t_{i-1} } ) -  \mus (X_{t_{i-1}} +\tWp_t - \tWp_{t_{i-1}})  \bigr)\, dt \Bigr|^2\Big] \\
			& \qquad \qquad \ge  \frac{c_1}{(-\log_2(t_i - t_{i-1}) + 1)^{2\beta}}(t_i - t_{i-1})^{2+\alpha} - c_2 e^{-c_3 n}.
		\end{aligned}
	\end{equation}
\end{lemma}

\begin{proof}
	Let $n\in\N$, let $\tau=\{t_1,\dots,t_{5n}\}\in\tPi_n$ with $0<t_1<\dots <t_{5n}=1$ and let $i\in \{1,\dots, 5n\}$ with $t_{i-1} \ge 1/2$.
	Throughout this proof, $c_1,c_2,\dots\in (0,\infty)$ denote positive constants, 
	which neither depend on $n$ nor on $ \tau$ nor on $i$.

	Since $\mu = \mus$ is bounded, see Lemma \ref{lemSob1}, we can use ~\cite[Theorem 1]{QZ02} to derive, similar to \eqref{rsv1NEW}, that
	\begin{equation}\label{rsv1_Sob}
		\begin{aligned}
			&	 \EE\Bigl[\Bigl| \int_{t_{i-1}}^{t_i} \bigl (\mus (X_{t_{i-1} } +W_t - W_{t_{i-1} } ) -  \mus(X_{t_{i-1}} + \tWp_t - \tWp_{t_{i-1}})  \bigr)\, dt \Bigr|^2\Big] \\
			& \qquad \qquad \ge c_1 \int_{0}^{2\pi} \EE\Bigl[\Bigl| \int_{t_{i-1}}^{t_i} \bigl (\mus (x +W_t - W_{t_{i-1} } ) -  \mus (x +\tWp_t - \tWp_{t_{i-1}})  \bigr)\, dt \Bigr|^2\Big] dx.
		\end{aligned}
	\end{equation}

Define $\fs\colon\R\to\R$ by
\begin{equation}\label{muSobx}
	\fs(x) = \sum_{j = 1}^\infty {j^{-\beta}}2^{-\alpha j}  \sin(2^j x), \qquad x \in \R,
\end{equation}
and note that $\mus = 1_{[-2\pi, 4\pi]} \fs$. Moreover, for all $x \in [0, 2\pi]$ and all $y,z\in [-2\pi,2\pi]$ we have
\begin{equation}\label{cvbn}
\mus(x+z) - \mus(x+y) = \fs(x+z) - \fs(x+y).
\end{equation}

Put  
	\[
B= \Bigl\{\,\sup_{t \in [t_{i-1}, t_i]} |W_t - W_{t_{i-1}}| \le 2\pi\Bigr\} \cap \Bigl\{\,\sup_{t \in [t_{i-1}, t_i]} |\tWp_t - \tWp_{t_{i-1}}| \le 2\pi\Bigr\}.
\]
%T Schlechte Notation pi und pi
%L yes, we should change it in the journal version
Then, by \eqref{cvbn} and the boundedness of $\mus$,
	\begin{equation}\label{nccc3}
		\begin{aligned}
			& \int_{0}^{2\pi} \EE\Bigl[\Bigl| \int_{t_{i-1}}^{t_i} \bigl (\mus (x +W_t - W_{t_{i-1} } ) -  \mus (x +\tWp_t - \tWp_{t_{i-1}})  \bigr)\, dt \Bigr|^2\Big] dx
		 \\
			& \qquad \ge \int_{0}^{2\pi} \EE\Bigl[1_{B}\Bigl| \int_{t_{i-1}}^{t_i} \bigl (\mus (x +W_t - W_{t_{i-1} } ) -  \mus(x +\tWp_t - \tWp_{t_{i-1}})  \bigr)\, dt \Bigr|^2\Big] dx \\
			& \qquad  \ge \int_{0}^{2\pi} \EE\Bigl[\Bigl| \int_{t_{i-1}}^{t_i} \bigl (\fs (x +W_t - W_{t_{i-1} } ) -  \fs (x +\widetilde W_t - \widetilde W_{t_{i-1}})  \bigr)\, dt \Bigr|^2\Big] dx - c_2 \PP(B^c).
		\end{aligned}
	\end{equation}

	Using standard results for the Brownian motion 
	 and \eqref{ndisc1}   we get
	\begin{equation}\label{mmm3}
		\begin{aligned}
	\PP(B^c) & \le 2\PP\Bigl(\sup_{t \in [t_{i-1}, t_i]} |W_t - W_{t_{i-1}}| >2 \pi\Bigl) = 2\PP\Bigl( \sup_{t \in [0,1]} |W_t| > \frac{2\pi}{\sqrt{t_i - t_{i-1}}}\Bigl) \\
	%L corrected, please check
	%T wrong correction
	& \le 8 \PP (W_1 > 4\pi \sqrt{n}) 
	%	& \le 4 \PP (W_1 > \pi/\sqrt{n}) 
	\le c_3 e^{-c_4 n}.
	\end{aligned}
\end{equation}

Similar to~\eqref{similar} we have for every $x\in\R$,
\begin{equation}\label{represent}
\fs(x) = \sum_{j\in \Z\setminus \{0\}} \frac{\sgn(j) |j|^{-\beta} 2^{- \alpha |j|} }{ 2 \bi}\exp\bigl(\bi \sgn(j) 2^{|j|}x\bigr).
\end{equation}
Since $\mus$ and $\fs$ are measurable and  bounded and $\fs$ is $2\pi$-periodic  we may thus apply Lemma~\ref{NEWx} to obtain
\begin{equation}\label{ccvv23}
	\begin{aligned}
	&	\EE\Bigl[\Bigl| \int_{t_{i-1}}^{t_i} \bigl (\fs (X_{t_{i-1} } +W_t - W_{t_{i-1} } ) -  \fs (X_{t_{i-1}} +\tWp_t - \tWp_{t_{i-1}})  \bigr)\, dt \Bigr|^2\Big] \\
	& \qquad\qquad  \qquad\qquad \ge c_5 (t_i - t_{i-1})^{2} \sum_{j \in \Z} | \widehat{ ( \fs ) }_j|^2 D(j \sqrt{t_i - t_{i-1}})\\
	& \qquad\qquad  \qquad\qquad = c_5 (t_i - t_{i-1})^{2} \sum_{j \in \Z\setminus\{0\}}\frac{|j|^{-2\beta}2^{-2\alpha |j|}}{4} D(\sgn (j) 2^{|j|}\sqrt{t_i - t_{i-1}})\\
	& \qquad\qquad  \qquad\qquad \geq c_6 (t_i - t_{i-1})^{2} (j^*)^{-2\beta}2^{-2\alpha j^*}  D(2^{j^*}\sqrt{t_i - t_{i-1}})
\end{aligned}
\end{equation}
with $D\colon\R\to\R$ given by~\eqref{Axx} and $j^\ast  =  \lceil -\log_2(\sqrt{t_i - t_{i-1}}) \rceil$. We conclude as in \eqref{similar2} that
\begin{equation}\label{dft2}
			 (j^\ast)^{-2\beta}2^{- 2\alpha j^\ast} D(2^{j^*}\sqrt{t_i - t_{i-1}}) 
\ge \frac{c_7}{(-\log_2(t_i - t_{i-1}) + 1)^{2\beta}} (t_i - t_{i-1})^{\alpha}.
\end{equation}	

Combining ~\eqref{rsv1_Sob} with \eqref{nccc3}, \eqref{mmm3}, \eqref{ccvv23} and \eqref{dft2} completes the proof of the lemma.

\end{proof}

\begin{proof}[Proof of Theorem~\ref{thm3}]

	Let $\alpha\in (0,1)$ and let $\beta\in (0,\infty)$. By Lemma \ref{lemSob1} we have that $\mus$ is bounded and satisfies $\mus \in C^\alpha(\R)\cap L^1(\R)$. 
	We may thus use Lemma \ref{lemf3NEW}, Lemma \ref{lemf4NEW} and Lemma~\ref{lemf6NEW} as well as ~\eqref{ndisc1} and \eqref{ndisc3} to derive, similar to~\eqref{thm2_2NEW} and \eqref{aux11NEW} that there exist $n^\ast\in \N$ and  $c_1,\ldots ,c_6\in (0,\infty)$ such that for all $n\in\N$ with $n\ge n^\ast$ and all $\tau=\{t_1,\dots,t_{5n}\}\in\tPi_n$ with $0<t_1<\dots <t_{5n}=1$,

	\begin{equation}\label{iter2NEWx}
		\begin{aligned}
			\EE\bigl[|Y_{1} - \tYp_{1}|^2\bigr] & \ge c_2 \left(1 - \frac{c_1}{n}\right)^n  \sum_{i=r(\tau)+1}^n \frac{(t_i - t_{i-1})^{2 +\alpha}}{(-\log_2(t_i - t_{i-1}) + 1)^{2\beta}}
					 - 5n \cdot \frac{c_3}{n^{2 + \alpha(1+\alpha)}} \\
					&\ge  c_4 \left(1 - \frac{c_1}{n}\right)^n \cdot \frac{1}{(\log_2(4n) + 1)^{2\beta} n^{1 + \alpha}} 
					 - \frac{c_5}{n^{1 + \alpha(1+\alpha)}}\\
					 & \ge   \frac{c_6 }{(\ln(n + 1))^{2\beta} n^{1 + \alpha}},
			 \end{aligned}
			\end{equation}
		where $r(\tau)$ is the unique index in $\{1,\dots,5n\}$ such that $t_{r(\tau)}= 1/2$.
		
		Furthermore, by Lemma~\ref{lem1NEW} and Lemma~\ref{lem3NEWa} we derive, similar to~\eqref{aux12NEW}, that for every 
		$p\in [1,\infty)$ 
		there exists $c\in (0,\infty)$ such that  for every $n\in\N$ and every $\tau\in \tPi_n$,
		\begin{equation}\label{aux12NEWx}
			\begin{aligned}
				\bigl(\EE\bigl[|Y_{1} - \tYp_{1}|^p\bigr] \bigr)^{1/p} \le    \frac{c}{n^{(1+\alpha)/2}}.
			\end{aligned}
		\end{equation}
	
	By~\eqref{iter2NEWx} and ~\eqref{aux12NEWx} we may apply Lemma~\ref{auxNEW} to obtain that there exists $n^\ast\in \N$ and for every $p\in[1,2]$ and $q\in (1,\infty)$ there exists $c\in (0,\infty)$ such that for all $n\in\N$ with $n\ge n^\ast$ and all $\tau\in\tPi_n$,
				\begin{equation}\label{NEWxx}
	\bigl(	\EE\bigl[|Y_{1} - \tYp_{1}|^p\bigr]\bigr)^{1/p}  \ge	\frac{c}{(\ln(n + 1))^{\beta\cdot 2q/p} \, n^{(1 + \alpha)/2}}.	
	\end{equation}
		
	For $p\in [1,\infty)$ and $\varepsilon \in (0,\infty)$ let $p^\ast = \min(2,p)$, choose $\teps\in (0,\infty)$ such that 
	\[
2\teps /p^\ast+	2\teps /(p^\ast)^2 + 2\teps^2/p^\ast \le \eps
	\]
	and let $\beta = 1/p^\ast +\teps$ and $q= 1+\teps$.  Then $\beta\in (1/p^\ast,\infty)$ and by Lemma~\ref{lemSob1} we have $\mus\in \cap_{\tilde p \geq p^\ast}  \mathsf{W}^{\alpha,\tilde p}(\R)$. Since $\beta \cdot 2q/p^\ast = 2/(p^\ast)^2 + 2\teps /p^\ast+	2\teps /(p^\ast)^2 + 2\teps^2/p^\ast \le 2/(p^\ast)^2 + \eps$ we conclude by~\eqref{NEWxx}  that there exist $n^\ast\in \N$ and  $c\in (0,\infty)$ such that for all $n\in\N$ with $n\ge n^\ast$ and all $\tau\in\tPi_n$,
			\begin{equation}\label{NEWxxx}
		\bigl(	\EE\bigl[|Y_{1} - \tYp_{1}|^p\bigr]\bigr)^{1/p}	  \ge	\frac{c}{(\ln(n+1))^{2/(p^\ast)^2+\eps}\, n^{(1 + \alpha)/2}}.
	\end{equation}
  Next, use ~\eqref{ndisc2} as well as Lemma~\ref{lemf1} and Lemma~\ref{lemf2} to obtain that for every $p\in [1,\infty)$  there exists $c\in (0,\infty)$ such that for all $n\in\N$,
  \begin{equation}\label{end1x}
  	\begin{aligned}
  		\inf_{\tau\in\mathcal T_n} e_p(\tau) & \ge 	\inf_{\tau\in\mathcal T_{\max(n,n^\ast)}}e_p(\tau) \\
  		 &\ge 	\inf_{\tau\in\widetilde{\mathcal T}_{\max(n,n^\ast)}}e_p(\tau)\\
  		&  \ge \frac{1}{2} \inf_{\tau\in\widetilde{ \mathcal T}_{\max(n,n^\ast)}} \bigl( \EE\bigl[|X_1-\tX^{\tau}_1|^p\bigr] \bigr)^{1/p}\\
  		& \ge c \inf_{\tau\in\widetilde{ \mathcal T}_{\max(n,n^\ast)}} \bigl(\EE\bigl[|Y_1-\tY^{\tau}_1|^p\bigr]  \bigr)^{1/p}.
  	\end{aligned}	
  \end{equation} 
  Combining ~\eqref{NEWxxx} with \eqref{end1x} completes the proof of Theorem ~\ref{thm3}.

\end{proof}

\section*{APPENDIX}

%T Neu
\begin{lemma}\label{auxNEW0}
	Let $B=(B_t)_{t\in[0,1]}$ be a Brownian bridge and define a real-valued stochastic process $Z=(Z_t)_{t\in[0,1]}$ by
	\[
	Z_t = B_t - \EE\Bigl[B_t\,\Bigr|\, \int_0^1 B_s\, ds\Bigr],\quad t\in [0,1].
	\]
	Then, for every $\delta\in (0,1/2)$ and all $0 < u <v<\infty$, 
	\[
	\PP \bigl( \forall t\in [\delta,1-\delta]\colon Z_t\in [u,v]\bigr) >0.
	\]
\end{lemma}

\begin{proof}
Fix $\delta\in (0,1/2)$ and let 
\[
Z^\delta = (Z^\delta_t= Z_t)_{t\in[\delta,1-\delta]}
\]
denote the restriction of $Z$ to the interval $[\delta,1-\delta]$. Note that $Z^\delta$ is a continuous, real-valued centered Gau{\ss}ian process. Hence, for every $\eps\in (0,\infty)$, 
\begin{equation}\label{rkhs1}
	\PP(\|Z^\delta\|_\infty <\eps) >0,
\end{equation}
see e.g. the proof of Lemma 5.1 in \cite{VZ2008}. Let $H(Z^\delta)$ denote the Reproducing Kernel Hilbert Space (RKHS) associated with the process $Z^\delta$. Note that, for any $h\in H(Z^\delta)$, the distributions $\PP^{Z^\delta}$ and $\PP^{Z^\delta+h}$ on the Borel $\sigma$-field in $C([\delta,1-\delta]; \R)$ are equivalent, see e.g. \cite[Lemma 3.2]{VZ2008}. Below we show that  $H(Z^\delta)$ contains all constant functions. By the latter fact and \eqref{rkhs1} we may then conclude that for all $\eps\in (0,\infty)$ and all $c\in\R$ 
\begin{equation}\label{rkhs2}
\PP(\forall t\in [\delta,1-\delta]\colon Z^\delta_t\in (c-\eps,c+\eps))  = \PP(\|Z^\delta-c\|_\infty <\eps)  > 0,
\end{equation}
which yields the statement in the lemma.

It remains to prove that $1\in H(Z^\delta)$. Recall that the RKHS $(H(V),\langle \cdot, \cdot\rangle_{H(V)})$ of a real-valued centered Gau{\ss}ian process $V=(V_t)_{t\in T}$ on $(\Omega,\mathcal A,\PP)$, where $T$ is a non-empty set, consists of all functions
\begin{equation}\label{rkhs2a}
h^V_U\colon T\to \R,\quad t\mapsto \EE[UV_t],
\end{equation} 
with $U$ ranging over the first order chaos $\mathcal H_1(V)$ of $V$, i.e. the  closure of the set of all linear combinations $\sum_{j=1}^m \alpha_j V_{t_j} $ with $m\in\N$, $\alpha_1,\ldots,\alpha_m\in\R$, $t_1,\ldots,t_m\in T$, in the space $L^2(\Omega,\mathcal A ,\PP)$, and for all $U_1,U_2\in \mathcal H_{ 1}(V)$, 
\begin{equation}\label{rkhs2b}
	\langle h^V_{U_1},h^V_{U_2}\rangle_{H(V)} = \EE[U_1 U_2].
\end{equation}

Put
\[
J = \int_0^1 B_s\, ds.
\]
Clearly, $J\in \mathcal H_1(B)$ and, for all $t\in[0,1]$,
\[
h_J^B(t) = \int_0^1 \EE[B_tB_s]\, ds = \int_0^1 \min(s,t)(1-\max(s,t))\, ds = (1-t)t/2.
\]
Using \eqref{rkhs2a} and \eqref{rkhs2b} it is easy to see that 
\begin{equation}\label{rkhs3}
H(Z^\delta) = \{ h_{| [\delta,1-\delta] }  \,|\, h\in H(Z) \}
\end{equation} 
and 
\begin{equation}\label{rkhs4}
	H(Z) = \{ h^B_U \in H(B) \,|\, U\in \mathcal H_1(B)\text{ with } \EE(UJ)=0\} = \{h\in H(B)\,|\, 	\langle h,h^B_{J}\rangle_{H(B)} = 0 \}.
\end{equation}

It is well known that the RKHS associated with the Brownian bridge $B$ is given by
\begin{equation}\label{rkhs5}
	\begin{aligned}
		H(B)& = \bigl\{ f\in C([0,1];\R)\,|\, f\text{ is absolutely continuous with } \\
		& \qquad\qquad \qquad \qquad f(0)=f(1)=0 \text{ and } \int_0^1 (f'(x))^2\, dx < \infty\}
	\end{aligned}
\end{equation}
%SE: Label hinzugefügt
and, for all $f,g\in H(B)$,
\[
\langle f,g\rangle_{H(B)} = \int_0^1 f'(t)g'(t)\, dt.
\]
Hence, for all $h\in H(B)$, 
\[
\langle h,h^B_{J}\rangle_{H(B)} = 0 \Leftrightarrow  \int_0^1 h'(t) (1/2-t)\, dt = 0\Leftrightarrow  \int_0^1 h'(t) t\, dt = 0 \Leftrightarrow  \int_0^1 h(t) \, dt = 0,
\]
and using \eqref{rkhs3} and \eqref{rkhs4} we conclude that
\begin{align*}
H(Z^\delta) &= \Bigl\{h_{| [\delta,1-\delta] }  \,|\, h\in H(B) \text{ with }\int_0^1 h(t)\, dt = 0.\Bigr\}
\end{align*}

Let $\gamma \in (0,\infty) $ and define $h_\gamma\colon [0,1]\to\R$ by
\[
h_{ \gamma}(t) =\begin{cases}  -2\delta^{-1}\gamma t, &\text{ if }t\in [0,\delta/2],\\
	                             -\gamma + 2\delta^{-1}(1+\gamma)(t-\delta/2), &\text{ if }t\in [\delta/2,\delta],\\
	                             1, &\text{ if }t\in [\delta,1-\delta],\\
	                             1- 2\delta^{-1}(1+\gamma)(t-1+\delta), &\text{ if }t\in [1-\delta,1-\delta/2],\\
	                             -\gamma + 2\delta^{-1}\gamma( t - 1 + \delta/2), &\text{ if }t\in [1-\delta/2,1].
	                            \end{cases}
\]
Using \eqref{rkhs5} it is easy to check that $h_\gamma\in H(B)$ with $h_{| [\delta,1-\delta] } = 1$ and
\begin{align*}
\int_0^1 h_\gamma(t)\, dt & = 2\int_0^{\delta/2} h_\gamma(t)\, dt + 2\int_{\delta/2}^\delta h_\gamma(t)\, dt + \int_\delta^{1-\delta} h_\gamma(t)\, dt \\
& = -\delta\gamma/2   +2(- \delta\gamma/2  + \delta(1+\gamma)/4) + (1-2\delta)= 1-3\delta/2-\delta\gamma.
\end{align*}
Choosing $\gamma= (1-3\delta/2)/\delta$ yields $\int_0^1 h_\gamma(t)\, dt=0$, which completes the proof.
\end{proof}

\begin{lemma}\label{auxNEW}
Let $c_1,r_1,r_2\in (0,\infty)$ with $r_1\le r_2$  and let $Z$ be a real-valued random variable such that $(\EE[Z^2])^{1/2} \ge c_1 r_1$ and for every 
$p\in [1,\infty)$
there exists $c_2(p)\in (0,\infty)$ such that 	$(\EE[|Z|^p])^{1/p} \le c_2(p) r_2$.  Then, for every $p\in [1,2]$ and every $q\in (1,\infty) $,
\[
(\EE[|Z|^p])^{1/p} \ge \Bigl(\frac{c_1}{c_2(\gamma)}\Bigr)^{2q/p} c_2(\gamma)  \Bigl(\frac{r_1}{r_2}\Bigr)^{2q/p}r_2
\]
with $\gamma = (2 - p / q) \cdot  q / (q - 1)\in [1,\infty)$. 
	\end{lemma}

\begin{proof}
Let $p\in [1,2]$ 
and let	$q \in (1, \infty)$. By the H\"older inequality,
\begin{align*}
	c_1^2 r_1^2 &  \le \EE\bigl[Z^2\bigr] = \EE\bigl[|Z|^{p/q}\cdot |Z|^{2-p/q}\bigr] 
		\le \bigl(\EE\bigl[|Z|^p\bigr]\bigr)^{1/ q} \cdot \bigl(\EE\bigl[|Z|^\gamma \bigr]\bigr)^{(q - 1)/ q}\\
		&  \le \bigl(\EE\bigl[|Z|^p\bigr]\bigr)^{1/ q} \cdot  (c_2(\gamma))^{2-p/q} r_2^{2-p/q}
\end{align*}
and therefore
\[
\bigl(\EE\bigl[|Z|^p\bigr]\bigr)^{1/ q} \ge  \Bigl(\frac{c_1}{c_2(\gamma)}\Bigr)^2c_2(\gamma)^{p/q} \Bigl(\frac{r_1}{r_2}\Bigr)^{2} 
r_2^{p/q},
\]
which completes the proof of the lemma. 
\end{proof}

\begin{proof}[Proof of Lemma \ref{lemSob1}]
	Let 
	$\beta \in(0, \infty)$. 
	Recall the definition~\eqref{muSobx} of the function 
	$\fs$
	and 
	note
	that 
	\begin{equation}\label{l1}
		\mus=1_{[-2\pi, 4\pi]}\cdot \fs.
	\end{equation}
	
	Clearly, for all $x\in\R$,
	\begin{equation}\label{l4}
		|\mus(x)|\leq |\fs(x)|\leq \sum_{j = 1}^\infty 2^{-\alpha j} <\infty,
	\end{equation}
	which implies (i).
	
	We next prove (ii). First, we show that $\fs \in C^\alpha(\R)$. To this end, we
	proceed similarly to the proof of~\cite[Theorem 4.9 in Chapter II]{Zyg2002}.  
	For $h \in (-1,1)\setminus\{0\}$ put $j_h=\lceil \log_2(1/|h|) \rceil$. Then there exists $c\in(0,\infty)$ such that for all $x \in \mathbb{R}$ and all $h \in (-1,1)\setminus\{0\}$ we have
	\begin{equation}\label{lemSob1_1}
		\begin{aligned}
			|\fs(x + h) - \fs(x)| &= \Bigl|\sum_{j = 1}^\infty j^{-\beta} 2^{-\alpha j} \bigl(\sin(2^j (x + h)) - \sin(2^j x)\bigr)\Bigr| \\ 
			&\le \sum_{j = 1}^{j_h} 2^{-\alpha j} \bigl|\sin(2^j (x + h)) - \sin(2^j x)\bigr| + 2 \sum_{j = j_h + 1}^\infty 2^{-\alpha j} \\
			&\le \sum_{j = 1}^{j_h} 2^{(1 - \alpha)j} |h| + 2 \sum_{j = j_h + 1}^\infty 2^{-\alpha j} \\
			&= 2^{1 - \alpha}|h| \cdot \frac{2^{(1 - \alpha)j_h} - 1}{2^{1-\alpha} - 1} + 2\cdot \frac{2^{-\alpha(j_h + 1)}}{1 - 2^{-\alpha}} \\
			&\le 2|h|\cdot \frac{2^{(1 - \alpha)(\log_2(1/|h|)  + 1)} }{2^{1-\alpha} - 1} + \frac{2}{1 - 2^{-\alpha}} \cdot 
			2^{-\alpha(\log_2(1/|h|)  + 1)}\\
			& = |h|^\alpha\Bigl( \frac{2^{2 - \alpha}}{2^{1-\alpha} - 1} + \frac{2^{1-\alpha}}{1-2^{-\alpha}} \Bigr)\\
			& \le c |h|^\alpha.
		\end{aligned}
	\end{equation}
	Since $\fs$ is bounded, see ~\eqref{l4},  we may thus conclude that
	$\fs \in 	C^\alpha(\R)$.
	Observe that  $\fs(-2\pi)=\fs(4\pi)=0$ and thus $\mus$ is continuous. Hence, $\mus \in C^\alpha(\R)$ by the construction of $\mus$.

	Since  $\mus$ is continuous and has compact support we conclude that $\mus \in L^q(\R)$ for all $q\geq 1$, which shows (iii).

	Finally, we prove (iv). Let $p\in[1,2]$ and $\beta\in(1/p, \infty)$.  
	Below
	we show that
	\begin{equation}\label{l2}
		\mus \in\mathsf W^{\alpha,p}(\R).
	\end{equation}
	Using (ii) 
	and ~\eqref{l2}
	we obtain that  for all $q\geq p$ there exists $c\in(0, \infty)$ such that
	\begin{align*}
		\int_\R\int_\R \frac{|\mus(x)-\mus(y)|^q}{|x-y|^{1+\alpha q}}\, dx\, dy &=\int_\R\int_\R \frac{|\mus(x)-\mus(y)|^p}{|x-y|^{1+\alpha p}}\cdot \Bigl(\frac{|\mus(x)-\mus(y)|}{|x-y|^{\alpha}}\Bigr)^{q-p}\, dx\, dy\\
		&\leq c \int_\R\int_\R \frac{|\mus(x)-\mus(y)|^p}{|x-y|^{1+\alpha p}}\, dx\, dy
		< \infty,
	\end{align*}
	and hence   $\mus\in\mathsf W^{\alpha,q}(\R)$ for all $q\geq p$.

	For the proof of \eqref{l2}
	we consider the function
	\[
	\nu=1_{[0, 2\pi]}\cdot \fs.
	\]
	We show below that
	\begin{equation}\label{l3}
		I =\int_{-2\pi}^{4\pi}\int_{-2\pi}^{4\pi} \frac{|\nu(x) - \nu(y)|^p}{|x-y|^{1+\alpha p}} \, dy \, dx < \infty.
	\end{equation}
	Applying  
	\cite[Lemma 5.1]{DiNezzaPalatucciValdinoci2012} with $\Omega=(-2\pi, 4\pi)$, $u=\nu_{|\Omega}$ and $K=[0, 2\pi]$
	we conclude that $\nu \in\mathsf W^{\alpha,p}(\R)$. Hence, also $\nu(\cdot +2\pi), \nu(\cdot -2\pi)\in W^{\alpha,p}(\R)$.
	Finally, using the fact that $\fs(0)=\fs(2\pi)=0$ and the $2\pi$-periodicity of $\fs$ we obtain 
	\[
	\mus=1_{[-2\pi, 0]}\cdot \fs+1_{[0, 2\pi]}\cdot \fs+1_{[2\pi, 4\pi]}\cdot \fs=\nu(\cdot+2\pi)+\nu+ \nu(\cdot-2\pi),
	\]
	which yields \eqref{l2}.

	It remains to prove \eqref{l3}.
	Clearly,
	\[
	I=I_1+2I_2,
	\]
	where
	\[
	I_1=\int_{0}^{2\pi}\int_{0}^{2\pi} \frac{|\fs(x) - \fs(y)|^p}{|x-y|^{1+\alpha p}} \, dy \, dx, \quad I_2=\int_{0}^{2\pi}\int_{[-2\pi, 0] \cup [2\pi, 4\pi] } \frac{|\fs(x)|^p}{|x-y|^{1+\alpha p}} \, dy \, dx.
	\]
	
	Using the fact that $\fs(0)=\fs(2\pi)=0$ and 
	$\fs\in C^\alpha(\R)$
	  we obtain that there exists $c\in(0, \infty)$ such that for all $x \in [0, 2\pi]$,
	\[
	|\fs(x)|=|\fs(x) - \fs(0)| = |\fs(x) - \fs(2\pi)| \le c \min(x^{\alpha}, (2\pi - x)^{\alpha}).
	\]
	Hence, there exists $c\in(0, \infty)$ such that
	\begin{equation*}
		\begin{aligned}
			I_2&=\frac{1}{\alpha p}\int_0^{2\pi} |\fs(x)|^p\cdot \Bigl(\frac{1}{x^{\alpha p}}-\frac{1}{(x+2\pi)^{ \alpha p}} + \frac{1}{(2\pi - x)^{\alpha p}}-\frac{1}{(4\pi-x)^{\alpha p}}\Bigr) \, dx \\
			&\le c \int_0^{2\pi}\min(x^{\alpha p}, (2\pi - x)^{\alpha p})\cdot \Bigl(\frac{1}{x^{\alpha p}} + \frac{1}{(2\pi - x)^{\alpha p}}\Bigr) \, dx
			\le 4\pi c < \infty.
		\end{aligned}
	\end{equation*}

	For the proof of $I_1 < \infty$ observe that
	\begin{equation*}
		\begin{aligned}
			I_1 &= \int_0^{2\pi} \int_{-x}^{2\pi - x} \frac{|\fs(x + y) - \fs(x)|^p}{|y|^{1+\alpha p}} \, dy \, dx \\
			&\le \int_{-2\pi}^{2\pi} \frac{1}{|y|^{1+ \alpha p}} \int_0^{2\pi} |\fs(x + y) - \fs(x)|^p \, dx \, dy.
		\end{aligned}
	\end{equation*} 
	Below we show that there exist $\delta \in (0, 1)$ and $c \in (0, \infty)$ such that for all $y \in (-\delta, \delta)\setminus\{0\}$,
	\begin{equation}\label{lemSob1_6}
		\begin{aligned}
			\int_0^{2\pi} |\fs(x + y) - \fs(x)|^2 \, dx \le  \frac{c|y|^{2\alpha }}{(-\log_2(|y|))^{2\beta}}.
		\end{aligned}
	\end{equation}
	Using
	the latter estimate,
	 the boundedness of $\fs$
	and the H\"older inequality we 
	obtain
	that	there exist $\delta \in (0, 1)$ and $c_1, \ldots, c_4\in(0, \infty)$ such that
	\begin{equation*}
		\begin{aligned}
			I_1 &\le c_1\int_{\delta}^{2\pi} \frac{1}{|y|^{1+ \alpha p}} dy + \int_{-\delta}^{\delta} \frac{1}{|y|^{1+ \alpha p}} \int_0^{2\pi} |\fs(x + y) - \fs(x)|^p \, dx \, dy \\
			&\le c_2 + c_3\int_{-\delta}^{\delta} \frac{1}{|y|^{1+ \alpha p}} \cdot \Bigl(\int_0^{2\pi} |\fs(x + y) - \fs(x)|^2 \, dx\Bigr)^{p/2} \, dy \\
			&\le c_2 + c_4 \int_{-\delta}^{\delta} \frac{1}{|y| \cdot(-\log_2(|y|))^{\beta p}} \, dy.
		\end{aligned}
	\end{equation*}
	Since $\beta p > 1$ we conclude that $I_1<\infty$.

	Next, we derive \eqref{lemSob1_6}.
	Using the representation ~\eqref{represent} of $\fs$
	we obtain by
	the Parseval's identity 
	that for all $y\in(-2\pi, 2\pi)$,
	\begin{equation}\label{l9}
		\begin{aligned}
			&\int_0^{2\pi} |\fs(x + y) - \fs(x)|^2 \, dx \\
			&\qquad \qquad =  \int_0^{2\pi} \Bigl|\sum_{j \in\Z\setminus\{0\}} \sgn(j)|j|^{-\beta}2^{- \alpha |j|}\cdot \frac{\exp(\bi \sgn(j) 2^{|j|}y) - 1}{2 \bi} \cdot \exp(\bi \sgn(j) 2^{|j|}x)\Bigr|^2 \, dx  \\
			& \qquad \qquad =  		\frac{\pi}{2} 
			 \sum_{j \in\Z\setminus\{0\}} a_j(y),
		\end{aligned}
	\end{equation}
	where 
	\[
	a_j(y)=|j|^{-2\beta}2^{- 2\alpha |j|}\cdot |\exp(\bi \sgn(j) 2^{|j|}y) - 1|^2
	\]
	for $j \in\Z\setminus\{0\}$ and $y\in(-2\pi, 2\pi)$.
	Clearly, there exists $c\in(0, \infty)$ such that for all $y \in (-1,1)\setminus\{0\}$,
	\begin{equation}\label{lemSob1_3}
		\begin{aligned}
			\sum_{ |j| \ge -\log_2(|y|)} a_j(y)
			&  \le 8 \sum_{j = \lceil -\log_2(|y|)\rceil}^\infty j^{-2\beta}2^{- 2\alpha j} \\
			&\le \frac{8}{(-\log_2(|y|))^{2\beta}} \cdot |y|^{2\alpha } \cdot \sum_{j=0}^\infty 2^{-2\alpha j}=\frac{c|y|^{2\alpha}}{(-\log_2(|y|))^{2\beta}}.
		\end{aligned}
	\end{equation}
	Moreover, using the inequality
	\[
	|e^{ix}-1|\leq |x|, \quad x\in\R,
	\]
	and the fact that there exists $\kappa\in\N$ such that the function 
	\[
	(0, \infty) \ni x \mapsto x^{-2\beta} 2^{2(1-\alpha)x}\in\R
	\]
	is monotonically increasing on $[\kappa, \infty)$, we obtain that there exists $c\in(0, \infty)$ such that for all $y \in (-1,1)\setminus\{0\}$,
	\begin{equation}\label{lemSob1_4}
		\begin{aligned}
			\sum_{0 < |j| < -\log_2(|y|)} a_j(y) &\le 2 \sum_{j=1}^{\lceil -\log_2(|y|) \rceil} j^{-2\beta} 2^{-2\alpha j} 2^{2j}|y|^2 \\
			& \le 2|y|^2 \sum_{j=1}^{\kappa} j^{-2\beta} 2^{2(1-\alpha)j} + 2|y|^2 \int_1^{\lceil -\log_2(|y|) \rceil+1} x^{-2\beta} 2^{2(1-\alpha)x} \, dx\\
			& \le c|y|^2  + 2|y|^2 \int_1^{ -\log_2(|y|) +2} x^{-2\beta} 2^{2(1-\alpha)x} \, dx.
		\end{aligned}
	\end{equation}
	Clearly,
	\begin{equation}\label{l7}
		\lim_{y \rightarrow 0} |y|^2 \Big/ \frac{|y|^{2\alpha}}{(-\log_2(|y|))^{2\beta}}=\lim_{y \rightarrow 0} |y|^{2(1-\alpha)} (-\log_2(|y|))^{2\beta}=0.
	\end{equation}
	Furthermore, by the rule of L'H\^opital, 
	\begin{equation}\label{l8}
		\begin{aligned}
			&\lim_{y \rightarrow 0} |y|^2 \int_1^{ -\log_2(|y|) + 2} x^{-2\beta} 2^{2(1-\alpha)x} \, dx \Big/ \frac{|y|^{2\alpha}}{(-\log_2(|y|))^{2\beta}} \\
			& \quad= \lim_{y \downarrow 0}  \int_1^{ -\log_2(y) + 2} x^{-2\beta} 2^{2(1-\alpha)x} \, dx \Big/ \frac{y^{2\alpha - 2}}{(-\log_2(y))^{2\beta}} \\
			&\quad= \lim_{y \downarrow 0}  \Biggl( \frac{-2^{4(1-\alpha)}}{\ln(2)}  \frac{1}{(-\log_2(y) + 2)^{2\beta}}y^{2\alpha -3}\Biggr) \Big/ \Biggl(\frac{(2\alpha - 2)(-\log_2(y)) + 2\beta /\ln(2)}{(-\log_2(y))^{2\beta+1}}y^{2\alpha-3}\Biggr) \\
			&\quad= \frac{2^{4(1-\alpha)}}{\ln(2)} \lim_{y \downarrow 0}  \frac{(-\log_2(y))^{2\beta+1}}{(-\log_2(y) + 2)^{2\beta}\cdot ((2-2\alpha)(-\log_2(y)) - 2\beta /\ln(2))}  \\
			&\quad = \frac{2^{4(1-\alpha)}}{\ln(2)(2 - 2\alpha)}.
		\end{aligned}
	\end{equation}
	Combining \eqref{l9} to \eqref{l8} yields \eqref{lemSob1_6}. This completes the proof of the lemma.
\end{proof}

\section*{Acknowledgement}
\noindent
The author Simon Ellinger is currently employed at the Deutsche Bundesbank. The views expressed in this paper are those of the three authors and do not necessarily reflect the views of the Deutsche Bundesbank.

\bibliographystyle{acm}
\bibliography{bibfile}

\end{document}